\documentclass[a4paper, 11pt, twoside, reqno]{amsart}

\usepackage{amsmath,amsfonts,amsthm,amsopn,xcolor,amssymb,enumitem}
\usepackage[colorlinks=true,urlcolor=blue,citecolor=blue,linkcolor=blue,linktocpage,pdfpagelabels,bookmarksnumbered,bookmarksopen]{hyperref}
\hypersetup{urlcolor=blue, citecolor=blue, linkcolor=blue}

\usepackage[left=2.61cm,right=2.61cm,top=2.72cm,bottom=2.72cm]{geometry}

\usepackage{color,epsfig,graphics}
\usepackage{latexsym}
\usepackage{graphicx}
\usepackage{mathrsfs}
\usepackage{bm}

\usepackage[T1]{fontenc}

\numberwithin{equation}{section}

\newtheorem{theorem}{Theorem}[section]
\newtheorem{lemma}[theorem]{Lemma}
\newtheorem{remark}[theorem]{Remark}
\newtheorem{proposition}[theorem]{Proposition}

\newtheorem*{claim*}{Claim}

\newcommand{\ef}{\eqref}
\newcommand{\dis}{\displaystyle}

\newcommand{\C}{\mathbb{C}}
\newcommand{\N}{\mathbb{N}}
\newcommand{\R}{\mathbb{R}}

\newcommand{\ep}{\varepsilon}

\newcommand{\intR}{\int_{\mathbb{R}^3}}

\newcommand{\IT}{\int_{\R^3}}
\newcommand{\LT}{L^2(\R^3)}

\title
[Schr\"odinger-Poisson system with a doping profile]
{Stable standing waves for Nonlinear Schr\"odinger-Poisson system 
with a doping profile}

\author[M.~Colin]{Mathieu Colin}
\author[T.~Watanabe]{Tatsuya Watanabe$^*$}

\address[M.~Colin]{\newline\indent
University of Bordeaux, CNRS, Bordeaux INP, 
IMB, UMR 5251,  F-33400, Talence, France
\newline\indent
 IMB, UMR 5251,  F-33400, Talence, France
}
\email{mathieu.colin@math.u-bordeaux.fr}

\address[T.~Watanabe]{\newline\indent 
Department of Mathematics, Faculty of Science, Kyoto Sangyo University,
\newline\indent
Motoyama, Kamigamo, Kita-ku, Kyoto-City, 603-8555, Japan}
\email{tatsuw@cc.kyoto-su.ac.jp}

\thanks{* Corresponding author}

\date{\today}

\begin{document}

\begin{abstract}
This paper is devoted to the study of 
the nonlinear Schr\"odinger-Poisson system with a doping profile.
We are interested in the existence of stable standing waves
by considering the associated $L^2$-minimization problem.
The presence of a doping profile causes a difficulty 
in the proof of the strict sub-additivity.
A key ingredient is to establish the strict sub-additivity
by adapting a scaling argument,
which is inspired by \cite{ZZou}.
When the doping profile is a characteristic function supported on 
a bounded smooth domain, smallness of some geometric quantity
related to the domain ensures the existence of stable standing waves.
\end{abstract}

\maketitle

\section{Introduction}

In this paper, we are concerned with the following nonlinear 
Schr\"odinger-Poisson system:
\begin{equation} \label{eq:1.1}
\begin{cases}
& -\Delta u+\omega u+e \phi u=|u|^{p-1} u \\
& -\Delta \phi =\frac{e}{2} \left( |u|^2- \rho(x) \right)
\end{cases} 
\quad \hbox{in} \ \R^3,
\end{equation}
where $\omega \in \R$, $e>0$ and $1<p<\frac{7}{3}$.
Equation \ef{eq:1.1} appears as a stationary problem for the 
time-dependent nonlinear Schr\"odinger-Poisson system:
\begin{equation} \label{eq:1.2}
\begin{cases}
i \psi_t + \Delta \psi - e\phi \psi + |\psi|^{p-1} \psi =0 
\quad \hbox{in} \ \R_+ \times \R^3, \\
-\Delta \phi = \frac{e}{2} \big( |\psi|^2-\rho(x) \big) 
\quad \hbox{in} \ \R_+ \times \R^3, \\
\psi(0,x)=\psi_0.
\end{cases}
\end{equation}
Indeed when we look for a standing wave of the form: $\psi(t,x)=e^{i \omega t} u(x)$,
we are led to the elliptic problem \ef{eq:1.1}.
In this paper, we are interested in the existence of stable standing waves 
for \ef{eq:1.2} by considering the solvability of the associated 
$L^2$-constraint minimization problem.

The Schr\"odinger-Poisson system appears 
in various fields of physics, such as quantum mechanics,
black holes in gravitation and plasma physics.
Especially, the Schr\"odinger-Poisson system plays an important role
in the study of semi-conductor theory; see \cite{Je, MRS, Sel},
and then the function $\rho(x)$ is referred as \textit{impurities} 
or a \textit{doping profile}.
The doping profile comes from 
the difference of the number densities of positively charged donor ions
and negatively charged acceptor ions,
and the most typical examples are characteristic functions, step functions 
or Gaussian functions.
Equation \ef{eq:1.1} also appears as a stationary problem
for the Maxwell-Schr\"odinger system.
We refer to \cite{BF, CW, CW2} for the physical background
and the stability result of standing waves for the Maxwell-Schr\"odinger system.
In this case, the constant $e$ describes the strength of the interaction
between a particle and an external electromagnetic field.

The nonlinear Schr\"odinger-Poisson system with $\rho \equiv 0$:
\begin{equation} \label{eq:1.3}
\begin{cases}
& -\Delta u+\omega u+e \phi u=|u|^{p-1} u \\
& -\Delta \phi =\frac{e}{2} |u|^2
\end{cases} 
\quad \hbox{in} \ \R^3
\end{equation}
has been studied widely in the last two decades.
Especially, the existence of $L^2$-constraint minimizers 
depending on $p$ and the size of the mass,
the existence of ground state solutions of \ef{eq:1.3}
and their stability have been investigated in detail; 
see \cite{AP, BJL, BS1, BS2, CDSS, CW, HLW, JL, K2, R, SaS, SS, ZZ}
and references therein.
On the other hand, the nonlinear Schr\"odinger-Poisson system with
a doping profile is less studied.
In \cite{DeLeo, DR}, the corresponding 1D problem has been considered.
Moreover, the linear Schr\"odinger-Poisson system 
(that is, the problem \ef{eq:1.1} without $|u|^{p-1}u$)
with a doping profile in $\R^3$ has been studied in \cite{Ben1, Ben2}.
As far as we know, there is no literature 
concerning with \ef{eq:1.1} and the existence of stable standing waves,
which is exactly the motivation of this paper.

To state our main results, let us give some notation. 
For $u \in H^1(\R^3,\C)$, 
the energy functional associated with \ef{eq:1.1} is given by
\begin{align}
I(u) &= \mathcal{E}(u) + \frac{\omega}{2} \intR |u|^2 \,dx, \notag \\
\mathcal{E}(u) &=\frac{1}{2} \intR |\nabla u |^2 \,dx 
-\frac{1}{p+1} \intR |u|^{p+1} \,d x + e^2A(u). \label{eq:1.4}
\end{align}
Here we denote the nonlocal term by $S(u)=S_1(u)+S_2$ with
\[
\begin{aligned}
S_1(u)(x)&:=(-\Delta)^{-1} \left( \frac{|u|^2}{2} \right) 
=\frac{1}{8 \pi |x|} *|u|^2, \\
S_2(x)&:= (-\Delta)^{-1} \left( \frac{- \rho}{2} \right) 
= -\frac{1}{8 \pi |x|} * \rho(x),
\end{aligned}
\]
and the functional corresponding to the nonlocal term by
\[
A(u) 
:= \frac{1}{4} \intR S(u) \big( |u|^2-\rho(x) \big) \,dx
= \frac{1}{32\pi} \intR \intR
\frac{\big( |u(x)|^2-\rho(x) \big) \big( |u(y)|^2-\rho(y) \big)}{|x-y|} \,dx\,dy.
\]
For $\mu>0$, let us consider the minimization problem:
\begin{equation} \label{eq:1.5}
\mathcal{C}(\mu)= \inf_{ u \in B(\mu) } \mathcal{E}(u),
\end{equation}
where $B(\mu) = \{ u \in H^1(\R^3,\C) \ ; \ \| u \|_{L^2(\R^3)}^2=\mu \}$.
We also define the set of minimizers by 
\[
\mathcal{M}(\mu) := \{ e^{i\theta}u(x) \ ; \ 
\theta \in [0,2\pi), u \in B(\mu), \ 
\mathcal{E}(u) = \mathcal{C}(\mu) \}.
\]
In this setting, the constant $\omega$ in \ef{eq:1.1} 
appears as a Lagrange multiplier.

Let us define the energy associated with \ef{eq:1.3}:
\[
E_{\infty}(u):=
\frac{1}{2} \| \nabla u \|_2 ^2
-\frac{1}{p+1} \intR |u|^{p+1} \,dx
+\frac{e^2}{4} \intR S_1(u) |u|^2 \,dx.
\]
Indeed if we assume $\rho(x) \to 0$ as $|x| \to \infty$,
\ef{eq:1.3} can be seen as a problem at infinity. 
We define the minimum energy associated with \ef{eq:1.3} by
\begin{equation*} 
0 \ge c_{e, \infty}(\mu) = c_{\infty}(\mu) := \inf_{u \in B(\mu)} E_{\infty}(u).
\end{equation*}
The existence of minimizers for $c_{e,\infty}(\mu)$ has been studied widely
and is summarized as follows.

\begin{enumerate}
\item[\rm(i)] In the case $2<p< \frac{7}{3}$, 
$c_{e,\infty}(\mu)$ is attained if $c_{e,\infty}(\mu)<0$.
Moreover $c_{e,\infty}(\mu)<0$ when 
$\mu$ is large for fixed $e$ or $e$ is small for fixed $\mu$.

\item[\rm(ii)] In the case $p=2$, 
$c_{e,\infty}(\mu)$ is attained if and only if $c_{e,\infty}(\mu)<0$.
Moreover $c_{e,\infty}(\mu)<0$ when $e$ is small for fixed $\mu$.

\item[\rm(iii)] In the case $1<p<2$, 
$c_{e,\infty}(\mu)<0$ for all $\mu$ and $e$.
Moreover $c_{e,\infty}(\mu)$ is attained 
when $\mu$ is small for fixed $e$ or $e$ is small for fixed $\mu$.

\end{enumerate}
For the proof, we refer to \cite{BS1, BS2, CDSS, CW, HLW, JL, K2, SS}. 
Especially when $2<p<\frac{7}{3}$, for fixed $e>0$,
there exists $\mu^*=\mu^*(e)>0$ such that 
\begin{equation} \label{eq:1.6}
\mu^* = \inf \{ \mu > 0 \mid c_{e, \infty}(\mu) < 0 \}.
\end{equation}
More precise information of $\mu^*$ has been obtained in e.g. \cite{CW, JL, SS}. 
It is also worth pointing out that as $e$ decreases, 
$\mu^*$ increases as well. 

For the doping profile $\rho$, we assume that
\begin{equation} \label{eq:1.7}
\rho(x) \in L^{\frac{6}{5}}(\R^3) \cap L^q_{loc}(\R^3) \ \text{for some} \ q>3 
\quad \mbox{and} \quad x \cdot \nabla \rho(x) \in L^{\frac{6}{5}}(\R^3), 
\end{equation}
\begin{equation} \label{eq:1.8}
\rho(x) \ge 0, \ \not\equiv 0 \quad \text{for} \ x \in \R^3,
\end{equation} 
\begin{equation} \label{eq:1.9}
\text{there exist $\alpha>2$ and $C>0$ such that} \ 
\rho(x) \le \frac{C}{1+|x|^{\alpha}} \quad \text{for} \ x \in \R^3.
\end{equation}
\begin{equation} \label{ASS}
\inf_{u \in H^1(\R^3,\C), \| u \|_{\LT}^2=1} 
\IT \left( |\nabla u|^2 + 2e^2 S_2(x) |u|^2 \right) dx = 0.
\end{equation}

The assumption \ef{eq:1.9} enables us to show that $S_2 \in L^{\infty}(\R^3)$
and decays to $0$ at infinity, 
which makes us to treat $S_2(x)$ as a trapping potential.
Typical examples are the Gaussian function $\rho(x) =\ep e^{-\alpha |x|^2}$
and $\rho(x) =\frac{\ep}{(1+|x|)^{\alpha}}$ for $\alpha> \frac{5}{2}$. 
The assumption \ef{ASS} means that a trapping effect caused by the doping profile $e^2 S_2(x)$
is too weak so that the operator $-\Delta + 2e^2S_2(x)$ does not have an eigenvalue,
which may happen if $e>0$ is sufficiently small.
\ef{ASS} will be used to establish the strict sub-additivity of $\mathcal{C}(\mu)$.

Our first main result is the following.

\begin{theorem}[Existence of a minimizer]  \label{thm:1.1}
Assume \ef{eq:1.7}-\ef{ASS} and let $e>0$ be fixed.
Suppose that $2<p<\frac{7}{3}$ and let $\mu > 2 \cdot 2^{\frac{1}{2p-4}} \mu^*$ be given.
Then there exists $\rho_0= \rho_0(e, \mu)>0$ such that 
if $\|\rho\|_{L^{\frac{6}{5}}(\R^3)}
+\|x \cdot \nabla \rho\|_{L^{\frac{6}{5}}(\R^3)} \le \rho_0$,
the minimization problem \ef{eq:1.5} admits a minimizer $u_{\mu}$.

Moreover the associated Lagrange multiplier $\omega=\omega(\mu)$ is positive.

\end{theorem}

The assumption \ef{eq:1.7} rules out the case 
$\rho$ is a characteristic function supported on a bounded smooth domain.
Even in this case, we are still able to obtain the existence of minimizers
under a smallness condition on some geometric quantity related to the domain;
See Section 6.

The positivity of the Lagrange multiplier $\omega(\mu)$ 
will be useful to establish the relation between $L^2$-constraint minimizers
and ground state solutions,
which we leave for a future work.
We also refer to \cite{CW3, DST} for this direction.

The next result states the orbital stability of standing waves 
corresponding to minimizers.

\begin{theorem}[Orbital stability of standing wave] \label{thm:1.2}
Under the assumption of Theorem \ref{thm:1.1}, 
the minimizer set $\mathcal{M}(\mu)$ 
is orbitally stable in the following sense:
For every $\varepsilon>0$, 
there exists $\delta(\varepsilon)>0$ such that
if an initial value $\psi_{0}$ satisfies 
$\dis \inf_{u \in \mathcal{M}(\mu)} \| \psi_{0} - u \|_{H^1(\R^3)} < \delta$, 
then the corresponding solution $(\psi, \phi)$ of \ef{eq:1.2} satisfies
\begin{align*}
\sup_{t>0} \inf_{y \in \R^3} \inf_{ u \in \mathcal{M}(\mu) }
\left\{ \big\| \psi(t,\cdot) - u (\cdot +y) \big\|_{H^1(\R^3)}
+\left\| \phi(t,\cdot) - \frac{e}{2} (-\Delta)^{-1} 
( |u(\cdot +y)|^2 - \rho ) \right\|_{D^{1,2}(\R^3)}
\right\} < \varepsilon.
\end{align*}
\end{theorem}

Here $D^{1,2}(\R^3, \R)=\dot{H}(\R^3, \R)$ 
denotes the completion of $C_0^{\infty}(\R^3, \R)$ with respect to 
the norm: $\| u\|_{D^{1,2}(\R^3)}^2=\IT |\nabla u|^2 \,dx$.
As for the global well-posedness of the Cauchy problem for \ef{eq:1.2}, 
see Section 4 below.

\smallskip
Here we briefly explain our strategy and its difficulty.
The existence of $L^2$-constraint minimizer can be shown
by applying the concentration compactness principle. 
A key of the proof is to establish the \textit{strict sub-additivity}:
\begin{equation} \label{eq:1.10}
\mathcal{C}(\mu)< \mathcal{C}\left(\mu^{\prime}\right)
+c_{\infty} \left(\mu-\mu^{\prime}\right)
\quad \hbox{for all} \quad 0 < \mu^{\prime}<\mu.
\end{equation}
In the case $2<p< \frac{7}{3}$ and $\rho \equiv 0$,
\ef{eq:1.10} can be readily obtained by adapting a suitable scaling.
However this scaling does not work straightforwardly 
if a doping profile is present, because of the loss of spatial homogeneity. 
In order to overcome this difficulty, 
we perform a scaling argument inspired by \cite{ZZou}.
By imposing the smallness of $\rho$ and $x \cdot \nabla \rho$,
it is possible to prove \ef{eq:1.10} if $2<p< \frac{7}{3}$.

\begin{remark} \label{rem:1.3}
{\rm(i)} \ It is natural to expect that whenever $c_{e,\infty}(\mu)$ is attained
for given $\mu$ and $e$, 
then \ef{eq:1.5} admits a minimizer if 
$\| \rho \|_{L^{\frac{6}{5}}(\R^3)} 
+ \| x \cdot \nabla \rho \|_{L^{\frac{6}{5}}(\R^3)}$ is sufficiently small.
However if we fix $e>0$ and take small $\mu>0$, 
$c_{e,\infty}(\mu)=0$ and $c_{e, \infty}(\mu)$ does not have a minimizer.
To obtain the existence of a minimizer
by imposing the smallness of $\rho$, 
in the process of the proof of \ef{eq:1.10},
we need to construct a test function whose energy is negative uniformly in $\rho$
even for small $\mu>0$.
For the moment, we don't know whether such a good test function exists
for small $\mu>0$,
even though by \ef{eq:1.8}, the presence of $\rho$ decreases the energy level 
of the truncated energy $E$ defined in \ef{eq:2.4} below.
Furthermore, our proof of \ef{eq:1.10} fails when $1< p \le 2$;
See Remark \ref{rem:3.5} below.

{\rm(ii)} We need to establish that
\begin{equation} \label{eq:1.12}
\mathcal{C}( \lambda \mu) < \lambda \mathcal{C}(\mu) \quad \text{fora all} \ \mu >0
\end{equation}
to obtain \ef{eq:1.10}.
When $\rho \equiv 0$ and $2<p< \frac{7}{3}$, 
\ef{eq:1.12} can be easily shown by a scaling argument.
However since the presence of $\rho(x)$ loses this good scaling property,
a careful analysis is required.
We can find that \ef{eq:1.12} holds for any $\lambda >1$ if $\mu$ is large; see Lemma \ref{lem:3.4}.
On the other hand, it is not straightforward to prove (1.12) when $\mu>0$ is small.
The assumptions \ef{ASS} and $\mu > 2 \cdot 2^{\frac{1}{2p-4}} \mu^*$ are crucially used in Lemma \ref{lem:3-ex}
to prove \ef{eq:1.12} for small $\mu>0$.
Once \ef{eq:1.12} is established for small $\mu>0$, 
we are able to obtain \ef{eq:1.10} no matter how $\mu'>0$ is small.

Furthermore if $p=2$, \ef{eq:1.3} is scaling invariant and hence \ef{eq:1.12} cannot be expected.
In fact, the closer $p$ is to 2, the larger the mass required for \ef{eq:1.12} to hold. 

We expect that our argument can be applied to other non-autonomous $L^2$-constraint minimization problems.
\end{remark} 

When $\rho$ is a characteristic function, 
further consideration is required
because $\rho$ cannot be weakly differentiable.
In this case, a key of the proof is the \textit{sharp boundary trace inequality}
which was developed in \cite{A}.
Then by imposing a smallness condition of some geometric quantity
related to the support of $\rho$, 
we are able to obtain the existence of stable standing waves.

\smallskip
This paper is organized as follows.
In Section 2, we introduce several properties of the energy functional
and some lemmas which will be used in this paper.
We establish the existence of a minimizer and prove Theorem \ref{thm:1.1}
in Section 3.
Section 4 is devoted to the solvability of the Cauchy problem, while
the stability of standing waves will be investigated in Section 5.
In Section 6, we consider the case $\rho$ is a characteristic function
and present the existence of standing waves for this case.
Finally in Section 7, we finish this paper   
by providing a concluding remark and one open question. 

Hereafter in this paper, 
unless otherwise specified, we write $\| u\|_{L^{p}(\R^3)}=\| u\|_p$.

\section{Variational formulation and preliminaries}

The aim of this section is to prepare several properties of the energy functional
and present intermediate lemmas which will be used later on.

\subsection{Reduction to a single equation} \ 

First we observe that the energy functional defined in \ef{eq:1.4} 
actually corresponds to the system \ef{eq:1.1}.
Let us consider the functional of two variables,
which is associated with \ef{eq:1.1}:
\[
J(u, \phi)
:=\frac{1}{2} \intR \left\{ |\nabla u|^2 + \omega |u|^2
+e \phi \left( |u|^2-\rho(x) \right)
-|\nabla \phi|^2 \right\} dx -\frac{1}{p+1} \intR |u|^{p+1} \,dx,
\]
for $(u,\phi) \in H^1(\R^3,\C) \times D^{1,2}(\R^3,\R)$.
A direct computation shows that the identity
$\frac{\partial J}{\partial u}=0$ is equivalent to the first equation of \ef{eq:1.1}.
Moreover one finds that
\[
\frac{\partial J}{\partial \phi} \psi
=\intR \left\{ -\nabla \phi \cdot \nabla \psi
+\frac{e}{2} \left( |u|^2- \rho(x) \right) \psi \right\} dx
\quad \hbox{for all} \ \psi \in C_0^{\infty}(\R^3,\R),
\]
from which we deduce that
\begin{equation} \label{eq:2.1}
-\Delta \phi =\frac{e}{2} \left(|u|^2-\rho(x) \right).
\end{equation}
Since $|u|^2-\rho \in L^{\frac{6}{5}}(\R^3)=(L^6(\R^3))^*$, 
by arguing similarly as in \cite{BF}, 
the Poisson equation \ef{eq:2.1} has a unique solution:
\[
\phi = e S(u)
= \frac{e}{2}(-\Delta)^{-1} \left( |u|^2-\rho(x) \right)
\in D^{1,2}(\R^3,\R).
\]
Moreover multiplying \ef{eq:2.1} by $\phi$, we have
\[
\intR |\nabla \phi|^2 \,dx =\frac{e}{2} \intR \phi \left( |u|^2 - \rho(x) \right) dx.
\]
This implies that
\[
\begin{aligned}
J(u,eS(u))&= 
\frac{1}{2} \intR (|\nabla u|^2+\omega |u|^2) \,dx 
-\frac{1}{p+1} \intR |u|^{p+1} \,dx
+\frac{e^2}{4} \intR S(u) \left( |u|^2-\rho(x) \right) dx \\
&= I(u).
\end{aligned}
\]

\subsection{Decomposition of the energy} \ 

In this subsection, we rewrite the energy functional $\mathcal{E}$
in a more convenient way.
First we decompose $\mathcal{E}$ in the following way:
\[
\begin{aligned}
\mathcal{E}(u)
&= \frac{1}{2} \intR |\nabla u|^2 \,dx -\frac{1}{p+1} \intR |u|^{p+1} \,dx, \\
&\quad +\frac{e^2}{4} \intR S_1(u) |u|^2 \,dx 
+\frac{e^2}{4} \intR S_2 |u|^2 \,dx 
-\frac{e^2}{4} \intR S_1(u) \rho(x) \,dx 
-\frac{e^2}{4} \intR S_2 \rho(x) \,dx 
\end{aligned}
\]
and define four nonlocal terms:
\begin{align} 
A_1(u)&=\frac{1}{4} \intR S_1(u) |u|^2 \,dx, \notag \\
A_2(u) &= -\frac{1}{4} \intR S_1(u) \rho(x) \,dx, \label{eq:2.2} \\
A_{2'}(u) &=\frac{1}{4} \intR S_2 |u|^2 \,dx, \notag \\
A_0 &=-\frac{1}{4} \intR S_2 \rho(x) \,dx. \notag
\end{align}
Note that $A_0$ is independent of $u$.
One can also see that
\[
A_2(u)= -\frac{1}{32\pi} \intR \intR \frac{|u(y)|^2 \rho(x)}{|x-y|} \,dx \,dy
= A_{2'}(u).
\]
Then we are able to write $\mathcal{E}$ in the following form:
\[
\mathcal{E}(u)= \frac{1}{2} \intR |\nabla u|^2 \,dx
-\frac{1}{p+1} \intR |u|^{p+1} \,d x 
+e^2 A_1(u) + 2e^2 A_2(u) + e^2 A_0.
\]
Recalling that 
\[
S_1(u)(x)=(-\Delta)^{-1} \left(\frac{|u|^2}{2}\right) \geq 0, 
\]
we find that  
\begin{equation} \label{eq:2.3}
A_1(u) \ge 0 \quad \hbox{for all} \ u \in H^1(\R^3,\C).
\end{equation}

Now it is convenient to put 
\begin{align} \label{eq:2.4}
E(u) &:=\mathcal{E}(u)-e^2 A_0 \notag \\
&= \frac{1}{2} \| \nabla u \|_2^2 - \frac{1}{p+1} \| u \|_{p+1}^{p+1}
+ e^2 A_1(u) + 2e^2 A_2(u).
\end{align}
Since $A_0$ is independent of $u$, 
we have only to consider the minimization problem for $E$.

\subsection{Estimates of nonlocal terms} \ 

This subsection is devoted to present
estimates for the nonlocal terms of the functional $E$. 
For later use, let us define
\[
A_3(u) = \frac{1}{2} \intR S_1(u) x \cdot \nabla \rho (x) \,dx 
\]
which is well-defined for $u \in H^1(\R^3,\C)$.
Then we have the following.

\begin{lemma} \label{lem:2.1}

For any $u \in H^1(\R^3,\C)$, 
$S_1$, $A_1$, $A_2$ and $A_3$ satisfy the estimates: 
\begin{align}
\| S_1(u) \|_{6} 
&\le C\| \nabla S_1(u) \|_{2} 
\le C \| u\|_{{\frac{12}{5}}}^2
\le C \| u \|_{2}^{\frac{3}{2}} \| \nabla u \|_{2}^{\frac{1}{2}}, \notag \\
0 \le A_1(u) &\leq C \| S_1(u) \|_6 \| u \|_{\frac{12}{5}}^2
\le C \| u \|_{2}^3 \| \nabla u \|_{2}, \notag \\
|A_2(u)| &\le \frac{1}{4} \|S_1(u) \|_{6} \| \rho \|_{\frac{6}{5}} 
\le  C \| \rho \|_{\frac{6}{5}} \| u \|_{\frac{12}{5}}^2 
\le C \| \rho \|_{\frac{6}{5}} 
\| u \|_{2}^{\frac{3}{2}} \| \nabla u \|_{2} ^{\frac{1}{2}}, \notag \\
|A_3(u)| &\le \frac{1}{2} \| S_1(u) \|_{6} 
\| x \cdot \nabla \rho \|_{\frac{6}{5}}
\le C \| x \cdot \nabla \rho \|_{\frac{6}{5}}
\| u \|_{2}^{\frac{3}{2}} \| \nabla u \|_{2}^{\frac{1}{2}}. \notag
\end{align}

\end{lemma}

For the proof of the inequality on $S_1(u)$, we refer to \cite{R}.
The other estimates can be obtained by the H\"odler inequality
and the Sobolev inequality.

\begin{lemma} \label{lem:2.2}
Assume that $u_n \rightharpoonup u$ in $H^1(\R^3)$. Then it follows that
\begin{align*}
\lim_{n \to \infty} \left\{ A_1(u_n-u) -A_1(u_n)+A_1(u) \right\} =0, \\
\lim_{n \to \infty} \left\{ A_2(u_n-u) -A_2(u_n)+A_2(u) \right\} =0.
\end{align*}
Moreover if $u_n \to u$ in $L^{\frac{12}{5}}(\R^3)$, we also have
\[
\lim_{n \to \infty} A_1(u_n) = A_1(u) \quad \hbox{and} \quad
\lim_{n \to \infty} A_2(u_n) = A_2(u).
\]
\end{lemma}

\begin{proof}
The proof for $A_1$ can be found in \cite[Lemma 2.2]{ZZ}.
Since $\rho \in L^{\frac{6}{5}}(\R^3)=(L^6(\R^3))^*$, 
the property for $A_2$ can be established in a similar way.
\end{proof}

\subsection{Scaling properties} \

In this subsection, we collect scaling properties of the nonlocal terms $A_1$ and $A_2$. 
For $a$, $b \in \R$ and $\lambda >0$, 
let us adapt the scaling $u_{\lambda} (x) := \lambda^a u\left( \lambda^b x\right)$. 
We first recall that 
\[
S_1(u)(x)  =(-\Delta)^{-1} \left(\frac{|u(x)|^2}{2} \right) 
=\frac{1}{8 \pi} \intR \frac{|u(y)|^2}{|x-y|} \,dy.
\]
Putting $y=\lambda^{-b} z$, we have 
\[
\begin{aligned}
S_1(u_{\lambda})(x)
&=\frac{1}{8 \pi } \intR \frac{|u_{\lambda}(y)|^2}{| x-y|} \,d y 
=\frac{\lambda^{2a}}{8 \pi } 
\intR \frac{|u(\lambda^b y)|^2}{|x- y|} \,dy \\ 
&= \frac{\lambda^{2a+b}}{8 \pi} \intR 
\frac{|u(\lambda^b y)|^2}{| \lambda^b x- \lambda^b y|} \,dy \\
&= \frac{\lambda^{2a-2b}}{8\pi} \intR 
\frac{|u(z)|^2}{|\lambda^b x -z|} \,dz. 
\end{aligned}
\]
Thus one finds that
\begin{align}
S_1(u_{\lambda})(x) &= \lambda^{2 a -2 b} S_1(u) ( \lambda^{b} x), \nonumber \\
A_1(u_\lambda) &=\lambda^{4 a-5 b} A_1(u),  \label{eq:2.5} \\
A_2( u_\lambda) 
& =-\frac{1}{4} \intR S_1(u_{\lambda}) \rho(x) \,d x 
=- \frac{\lambda^{2 a-2 b}}{4} \intR S_1(u)(\lambda^{b} x) 
\rho( x) \,dx \nonumber \\
&=- \frac{\lambda^{2 a-5 b}}{4} \intR S_1(u) \rho\left(\lambda^{-b} x\right) \,d x.
\label{eq:2.6}
\end{align}
By the H\"older inequality, it follows that
\begin{equation} \label{eq:2.7}
|A_2(u_{\lambda})| \le \frac{\lambda^{2a-5b}}{4} 
\| S_1(u)\|_6 \| \rho( \lambda^{-b} \cdot ) \|_{\frac{6}{5}}
\le C \lambda^{2a-\frac{5}{2}b} \| \rho \|_{\frac{6}{5}}
\| u \|_2^{\frac{3}{2}} \| \nabla u \|_2^{\frac{1}{2}}.
\end{equation}

\subsection{Nehari and Pohozaev identities} \ 

This subsection is devoted to establish the
Nehari identity and the Pohozaev identity associated with \ef{eq:1.1}.
First we observe that
\[
\begin{aligned}
I^{\prime}(u) \varphi
&=\intR \nabla u \cdot \nabla \bar{\varphi} \,dx
+ \omega \intR u \bar{\varphi} \,dx -\int |u|^{p-2}u \bar{\varphi} \,dx \\
&\quad +\frac{e^2}{4} \intR S^{\prime}(u) \varphi \left( |u|^2-\rho(x) \right) dx 
+\frac{e^2}{2} \intR S(u) u \bar{\varphi} \,dx,
\end{aligned}
\]
for any $\varphi \in H^1(\R^3, \C)$.
The definition $S(u)=S_1(u)+S_2$ shows that $S^{\prime}(u)= S_1'(u)$,
and moreover
\[
S_1^{\prime}(u) \varphi=(-\Delta)^{-1} *(u \bar{\varphi}).
\]
This yields that 
\[
S^{\prime}(u) u=S_1^{\prime}(u) u =(-\Delta)^{-1} *  |u|^2 =2 S_1(u)
= 2S(u) - 2S_2
\]
and hence 
\[
\begin{aligned}
&\frac{e^2}{4} \intR S'(u) u  \left( |u|^2- \rho(x) \right) dx
+\frac{e^2}{2} \intR S(u) |u|^2 \,dx \\
& =\frac{e^2}{2} \intR (S(u)-S_2) \left( |u|^2-\rho(x) \right) dx
+\frac{e^2}{2} \intR S(u) |u|^2 \,dx \\
& = e^2 \intR S(u) \left( |u|^2-\rho(x)\right) dx
-\frac{e^2}{2} \intR S_2 \left( |u|^2- \rho(x) \right) dx 
+\frac{e^2}{2} \int_R S(u) \rho(x) \,dx \\
& =e^2 \intR S(u) \left( |u|^2-\rho(x) \right) dx
+\frac{e^2}{2} \intR \left\{ S_1(u) \rho(x) -S_2 |u|^2+2 S_2 \rho(x) \right\} dx \\
& =e^2 \intR S(u) \left( |u|^2-\rho(x) \right) dx
+ e^2 \intR S_2 \rho(x) \,dx.
\end{aligned}
\]
Here we used the fact $\intR S_1(u) \rho \,dx = - \intR S_2 |u|^2 \,dx$.
Thus we find that the Nehari identity corresponding to \ef{eq:1.1} is given by
\begin{align} \label{eq:2.8}
0 &=\intR \left\{ |\nabla u|^2 +\omega |u|^2 -|u|^{p+1}
+e^2 S(u)\left(|u|^2-\rho(x)\right)
 +e^2 S_2 \rho(x) \right\} dx \notag \\
&= \| \nabla u \|_2^2 + \omega \| u \|_2^2 - \| u \|_{p+1}^{p+1} 
+4 e^2 A_1(u) + 4 e^2 A_2(u).
\end{align}

Next, we show that the Pohozaev identity associated with \ef{eq:1.1} is 
described as 
\begin{align} \label{eq:2.9}
0 &=\int_{\mathbb{R}^3}
\left\{\frac{1}{2}|\nabla u|^2+\frac{3\omega}{2} |u|^2
-\frac{3}{p+1} |u|^{p+1} 
+\frac{5 e^2}{4} S(u)\left(|u|^2-\rho(x) \right)
-\frac{e^2}{2} S(u) x \cdot \nabla \rho(x) 
\right\} dx \notag \\
&= \frac{1}{2} \| \nabla u \|_2^2 + \frac{3\omega}{2} \| u \|_2^2
- \frac{3}{p+1} \| u \|_{p+1}^{p+1} 
+5e^2 A_1(u) + 10 e^2 A_2(u) - e^2 A_3(u).
\end{align}
First we derive \ef{eq:2.9} by a formal calculation.
Let us consider $u_\lambda(x) =u\left(\frac{x}{\lambda}\right)$,
that is, take $a=0$ and $b=-1$. 
Then from \ef{eq:2.5} and \ef{eq:2.6}, one has
\[
\begin{aligned}
I(u_\lambda) &=
\frac{1}{2} \|\nabla u_{\lambda} \|_{2}^2
+\frac{\omega}{2} \| u_{\lambda} \|_{2}^2
-\frac{1}{p+1} \|u_\lambda \|_{p+1}^{p+1}
+e^2 A_1(u_{\lambda})+2 e^2 A_2( u_{\lambda})
+e^2 A_0 \\
&= \frac{\lambda}{2}\| \nabla u\|_{2}^2
+\frac{\lambda^3\omega }{2} \| u\|_{2}^2
-\frac{\lambda^3}{p+1}\|u\|_{p+1}^{p+1}
+e^2 \lambda^5 A_1(u)
-\frac{\lambda^5 e^2}{2} \int_{\R^3} S_1(u) \rho\left(\lambda x\right) \,dx
+e^2 A_0.
\end{aligned}
\]
Now we suppose that $u$ is a solution of \ef{eq:1.1}.
Since $A=A_1+2 A_2+A_0$, it follows that
\begin{align} \label{eq:2.10}
0&=\frac{d}{d \lambda} I(u_\lambda) \Big|_{\lambda=1} \\
&= \frac{1}{2}\|\nabla u \|_2^2
+\frac{3 \omega}{2}\| u\|_2^2
-\frac{3}{p+1}\|u\|_{p+1}^{p+1}
 +5 e^2 A_1(u) 
+ 10 e^2 A_2(u)
-\frac{e^2}{2} \int_{\mathbb{R}^3} S_1(u) x \cdot \nabla \rho(x) \,d x \notag \\
&=\frac{1}{2} \| \nabla u \|_2^2 
+\frac{3 \omega}{2}\|u\|_2^2
-\frac{3}{p+1}\| u\|\|_{p+1}^{p+1} 
+5 e^2 A(u) -5 e^2 A_0 \notag \\
&\quad -\frac{e^2}{2} \intR S(u) x \cdot \nabla \rho(x) \,d x
+\frac{e^2}{2} \intR S_2 x \cdot \nabla \rho(x) \,d x. \notag
\end{align}
We put 
\[
\begin{aligned}
R
&= -5 e^2 A_0 +\frac{e^2}{2} \intR S_2(x) x \cdot \nabla \rho(x) \,d x \\
&= \frac{5 e^2}{4} \intR S_2(x) \rho(x) \,d x
+\frac{e^2}{2} \intR S_2(x) x \cdot \nabla \rho (x) \,dx.
\end{aligned}
\]
Recalling that 
\[
-\Delta S_2=\frac{-\rho(x)}{2} \quad \hbox{and} \quad 
\intR \left| \nabla S_2 \right|^2 dx =- \frac{1}{2} \intR S_2 \rho(x) \,d x,
\]
one finds that
\[
\begin{aligned}
\intR S_2 x \cdot \nabla \rho \,d x 
&= -\intR \nabla S_2 \cdot x \rho \,d x -3 \intR S_2 \rho \,d x \\
&= -2 \intR \nabla S_2 \cdot x \Delta S_2 \,d x -3 \intR S_2 \rho \,d x \\
&= 2 \intR \nabla\left(\nabla S_2 \cdot x\right) \cdot \nabla S_2 \,d x
-3 \intR S_2 \rho \,d x \\
&= 2 \intR \left|\nabla S_2\right|^2 \,d x
+2 \intR x \cdot \nabla\left(\frac{1}{2}\left|\nabla S_2\right|^2\right) \,d x
-3 \intR S_2 \rho \,d x \\
&= 2 \intR \left|\nabla S_2\right|^2 \,d x
-3 \intR\left|\nabla S_2\right|^2 \,d x -3 \intR S_2 \rho \,d x \\
&= -\intR \left|\nabla S_2\right|^2 \,d x-3 \intR S_2 \rho \,d x 
= -\frac{5}{2} \intR S_2 \rho d x.
\end{aligned}
\]
This means that $R=0$.
Thus from \ef{eq:2.10}, we obtain \ef{eq:2.9}.

A rigorous proof can be done by establishing the $C^{1,\alpha}$-regularity of 
any weak solution of \ef{eq:1.1} for some $\alpha \in (0,1)$. 
Note that since $\rho \in L^q_{loc}(\R^3)$ for some $q>3$, 
it follows by the elliptic regularity theory that $S_2 \in W^{2,q}_{loc}(\R^3)
\hookrightarrow C^{1,\alpha}_{loc}(\R^3)$. 
The smoothness of $u$ can be shown similarly by applying the elliptic regularity theory.
Then multiplying $x \cdot \nabla \bar{u}$ and $e x \cdot \nabla S(u)$ by \ef{eq:1.1} 
respectively, 
integrating over $B_R(0)$ and passing to a limit $R \to \infty$,
we are able to prove \ef{eq:2.9} as in \cite{BL, Ca}.

\begin{lemma} \label{lem:2.3}
Any nontrivial solution $u$ of \ef{eq:1.1} satisfies the following identity.
\begin{align*}
(5 p-7) E(u) &= 
2(p-2) \|\nabla u\|_2^2 -\frac{(3 p-5) \omega}{2} \| u\|_2^2 
+ 8 e^2 A_2(u) - (3-p) e^2 A_3(u).
\end{align*}

\end{lemma}

\begin{proof}
From \ef{eq:2.8} and \ef{eq:2.9}, we find that
\[
\begin{aligned}
\frac{5 p-7}{p+1}\| u \|_{p+1}^{p+1}
&=3\|\nabla u\|_2^2 - \omega \| u\|_2^2 
-20 e^2 A_2(u) + 4e^2 A_3(u), \\
(5 p-7) e^2 A_1(u)
&=\frac{5-p}{2}\|\nabla u \|_2^2 -\frac{3(p-1) \omega}{2}\| u \|_2^2 
-2(5p-1)e^2 A_2(u) + (p+1)e^2 A_3(u).
\end{aligned}
\]
Thus one deduces that
\begin{align*}
(5 p-7) E(u)
&=\frac{5 p-7}{2}\|\nabla u\|_2^2 
-\frac{5 p-7}{p+1}\| u\|_{p+1}^{p+1} 
+(5p-7)e^2 A_1(u) + 2(5p-7) e^2 A_2(u) \\
&=2(p-2)\|\nabla u \|_2^2 -\frac{(3 p-5) \omega}{2} \| u \|_2^2
+8 e^2 A_2(u) - (3-p) e^2 A_3(u).
\end{align*}
This ends the proof.
\end{proof}

\subsection{Decay of $S_2(x)$ and $A_2$} \ 

In this subsection, we prove that the nonlocal term 
$S_2(x)= (-\Delta)^{-1}\left( \frac{- \rho}{2} \right)$
decays at infinity under the assumption \ef{eq:1.9}.

\begin{lemma}  \label{lem:2.4}
Assume \ef{eq:1.9}. 
Then it  holds that 
\begin{equation} \label{eq:2.11}
\lim_{|x| \to \infty} S_2(x) =0,
\end{equation}

\end{lemma}

\begin{proof}
For a fixed $x \in \R^3$, it follows by the definition of $S_2$ that
\[
|S_2(x)| \le \frac{1}{8 \pi} \intR \frac{| \rho(y)|}{|x-y|} \,dy
= \frac{1}{8 \pi} \int_{|y| \le \frac{1}{2}|x|} \frac{| \rho(y)|}{|x-y|} \,dy
+\frac{1}{8 \pi} \int_{|y| \ge \frac{1}{2}|x|} \frac{| \rho(y)|}{|x-y|} \,dy.
\]
By the assumption \ef{eq:1.9} and the fact $\alpha>2$,
one can take $q$ so that $\rho \in L^q(\R^3)$ and
$\max \{ 1, \frac{3}{\alpha} \} < q < \frac{3}{2}$.
This also implies that the H\"older conjugate $q'$ satisfies $q'>3$.
Then observing that
\[
|y| \le \frac{1}{2} |x| \quad \Rightarrow \quad 
|x-y| \ge |x| - |y| \ge \frac{1}{2}|x|
\]
and by the H\"older inequality, we have
\begin{align*}
\int_{|y| \le \frac{1}{2}|x|} \frac{| \rho(y)|}{|x-y|} \,dy
&\le \frac{2}{|x|} \int_{|y| \le \frac{1}{2} |x|} | \rho(y)| \,dy \\
&\le \frac{2}{|x|} 
\left( \int_{|y| \le \frac{1}{2} |x|} \,dy \right)^{\frac{1}{q'}}
\left( \int_{|y| \le \frac{1}{2} |x|} |\rho(y)|^q \,dy \right)^{\frac{1}{q}} \\
&\le \frac{C}{|x|^{1-\frac{3}{q'}}} \| \rho \|_{L^q(\R^3)} \to 0 
\quad \text{as} \ |x| \to \infty.
\end{align*}
Next we decompose
\[
\int_{|y| \ge \frac{1}{2} |x|} \frac{|\rho(y)|}{|x-y|} \,dy
= \int_{|y| \ge \frac{1}{2} |x|, |x-y| \le \frac{1}{2}|x|} \frac{|\rho(y)|}{|x-y|} \,dy
+ \int_{|y| \ge \frac{1}{2} |x|, |x-y| \ge \frac{1}{2}|x|} \frac{|\rho(y)|}{|x-y|} \,dy.
\]
Then from \ef{eq:1.9}, one finds that
\begin{align*}
\int_{|y| \ge \frac{1}{2} |x|, |x-y| \le \frac{1}{2}|x|} \frac{|\rho(y)|}{|x-y|} \,dy
& \le \int_{|y| \le \frac{1}{2}|x|, |x-y| \le \frac{1}{2}|x|}
\frac{1}{|x-y|} \cdot \frac{C}{1+|y|^{\alpha}} \,dy \\
&\le \frac{C}{|x|^{\alpha}} 
\int_{|x-y| \le \frac{1}{2}|x|} \frac{1}{|x-y|} \,dy
\le \frac{C}{|x|^{\alpha-2}} \to 0 \quad \text{as} \ |x| \to \infty.
\end{align*}
Moreover by the H\"older inequality, we also have
\begin{align*}
\int_{|y| \ge \frac{1}{2} |x|, |x-y| \ge \frac{1}{2}|x|} \frac{|\rho(y)|}{|x-y|} \,dy
&\le \left( \int_{|y| \ge \frac{1}{2}|x|} |\rho(y)|^q \,dy \right)^{\frac{1}{q}}
\left( \int_{|x-y| \ge \frac{1}{2}|x|} \frac{1}{|x-y|^{q'}} \,dy \right)^{\frac{1}{q'}} \\
&\le \frac{C}{|x|^{1-\frac{3}{q'}}} \| \rho \|_{L^q(\R^3)} 
\to 0 \quad \text{as} \ |x| \to \infty.
\end{align*}
Thus we obtain \ef{eq:2.11}.
\end{proof}

Now we recall that 
$A_2(u)= \frac{1}{4} \intR S_2(x) |u|^2 \,dx$.
By Lemma \ref{lem:2.4}
and the elliptic regularity theory,
it follows that $S_2 \in L^{\infty}(\R^3)$ and
$S_2(x) \to 0$ as $|x| \to \infty$.
Since $|u|^2 \in L^1(\R^3)$, 
we are able to apply the Lebesgue dominated convergence theorem to obtain
\begin{equation} \label{eq:2.12}
A_2 \big( u ( \cdot - k) \big) \to 0 \quad \text{as} \ |k| \to \infty
\ \text{for any} \ u \in H^1(\R^3, \C).
\end{equation}

\section{Existence of a minimizer}

In this section, we aim to prove that the minimization problem \ef{eq:1.1}
admits a solution, provided that the minimum energy for $\rho \equiv 0$ is negative
and $\| \rho\|_{\frac{6}{5}}+ \| x \cdot \nabla \rho \|_{\frac{6}{5}}$
is small. First we begin with the following.

\begin{lemma} \label{lem:3.1}
Suppose that $1<p< \frac{7}{3}$. 
Then for any $\mu>0$, 
$E$ is bounded from below on $B(\mu)$.
\end{lemma}

\begin{proof}
We use the fact that $A_1 \ge 0$ and $\frac{3}{2}(p-1)<2$.
The Gagliardo-Nirenberg inequality,  
the Young inequality and Lemma \ref{lem:2.1} yield that
\[
\begin{aligned}
E(u)
&=\frac{1}{2} \intR |\nabla u|^2 \,dx -\frac{1}{p+1} \intR |u|^{p+1} \,dx
+e^2 A_1(u)+2e^2 A_2(u) \\
& \geq \frac{1}{2}\|\nabla u\|_2^2 
- C \| u \|_2^{\frac{5-p}{2}} \| \nabla u \|_2^{\frac{3}{2}(p-1)} 
- C e^2 \|\rho\|_{\frac{6}{5}} \| u \|_2^{\frac{3}{2}} \| \nabla u \|_2^{\frac{1}{2}} \\
& \geq \frac{1}{2}\|\nabla u\|_2^2 
-\frac{3(p-1) \ep}{4}\|\nabla u\|_2^2
-\frac{(7-3p) C}{4 \ep^{\frac{3(p-1)}{7-3p}}} \| u \|_2^{\frac{2(5-p)}{7-3 p}}
-\frac{\varepsilon}{4}\|\nabla u\|_2^2 
-\frac{4 C e^{\frac{8}{3}}}{3 \ep^{\frac{1}{3}}} 
\|\rho\|_{\frac{6}{5}}^{\frac{4}{3}} \| u\|_2^2 \\ 
&\geq \frac{1}{4}\|\nabla u\|_2^2
-C \mu^{\frac{5-p}{7-3p}}- C e^{\frac{8}{3}} \mu \| \rho \|_{\frac{6}{5}}^{\frac{4}{3}}
\ge -C \mu^{\frac{5-p}{7-3p}}- C e^{\frac{8}{3}} \mu \| \rho \|_{\frac{6}{5}}^{\frac{4}{3}},
\end{aligned}
\]
from which we conclude.
\end{proof}

Next we define
\begin{equation} \label{eq:3.1}
c(\mu)= \inf_{\| u \|_2^2 = \mu} E(u).
\end{equation}
Instead of $C(\mu)$ defined in \ef{eq:1.5}, 
it suffices to show that $c(\mu)$ is attained because of \ef{eq:2.4}.

\begin{lemma} \label{lem:3.2}
Suppose that $1<p<\frac{7}{3}$. 
Then $c(\mu) \leq 0$ for all $\mu >0$.
\end{lemma}

\begin{proof}
Let us consider $u_{\lambda}(x)=\lambda^{\frac{3}{2}} u(\lambda x)$ 
for $\| u \|_2^2 = \mu$.
Note that $\| u_{\lambda} \|_2^2=\mu$ for any $\lambda >0$.
Using \ef{eq:2.5} and \ef{eq:2.7}, we have
\[
\begin{aligned}
E(u_{\lambda})
&= \frac{\lambda^2}{2}  \| \nabla u \|_2^2 
-\frac{\lambda^{\frac{3(p-1)}{2}}}{p+1} \| u \|_{p+1}^{p+1}
+e^2 \lambda A_1(u)+2e^2 A_2(u_{\lambda}) \\
&\leq \frac{\lambda^2}{2} \| \nabla u\|_2^2
-\frac{\lambda^{\frac{3(p-1)}{2}}}{p+1} \|u\|_{p+1}^{p+1}
+e^2 \lambda A_1(u) 
+ C e^2 \lambda^{\frac{1}{2}} \| \rho\|_{\frac{6}{5}} \| u \|_2^{\frac{3}{2}}
\| \nabla u \|_2^{\frac{1}{2}} 
\rightarrow 0 \text { as } \lambda \rightarrow 0+. 
\end{aligned}
\]
This implies that $c(\mu) \le 0$, as claimed.
\end{proof}

Next in order to prove the strict sub-additivity, 
we first mention that the condition \ef{eq:1.8} on $\rho$
and the definition of $A_2$ in \ef{eq:2.2} imply that $A_2(u) < 0$ 
for $u \not\equiv 0$ and hence
\[
\begin{aligned}
E(u) &= \frac{1}{2} \| \nabla u \|_2^2 -\frac{1}{p+1} \| u \|_{p+1}^{p+1}
+e^2 A_1(u) + 2 e^2 A_2(u) \\
&< \frac{1}{2} \| \nabla u \|_2^2 - \frac{1}{p+1} \| u \|_{p+1}^{p+1}
+e^2 A_1(u ) = E_{\infty}(u)
\quad \text{for any} \ u \in H^1(\R^3,\C) \setminus \{ 0 \}.
\end{aligned}
\]
Thus it holds that  
\begin{equation} \label{eq:3.2}
c(\mu) \le c_{\infty}(\mu) \quad \text{for all} \ \mu >0.
\end{equation}
Moreover 
since $c_{\infty}(\mu)$ is attained and negative 
if $2<p<\frac{7}{3}$ and $\mu > \mu^*$,
it follows that
\begin{equation} \label{eq:3.3}
c(\mu) \leq c_{\infty}(\mu) < \frac{1}{2} c_{\infty}(\mu)<0 \quad \hbox{if} \ \
2< p<\frac{7}{3} \ \hbox{and} \ \mu > \mu^*.
\end{equation} 

Using \ef{eq:3.2}, we are able to prove the following lemma.

\begin{lemma} \label{lem:3.3}
Suppose that $1<p<\frac{7}{3}$ and let $\mu >0$ be given.

\begin{enumerate}

\item[\rm(i)] For all $0< \mu' \le \mu$, 
it holds that $c(\mu) \le c(\mu')$.

\item[\rm(ii)] $c(\mu)$ is continuous with respect to $\mu>0$.

\item[\rm(iii)] It follows that $c( \lambda \mu) \le \lambda c_{\infty}(\mu)$
for all $\lambda >1$.

\end{enumerate}

\end{lemma}

\begin{proof}
(i) Let $0< \mu' < \mu$ be given.
For any $\ep>0$, there exist $u$, $v \in C_0^{\infty}(\R^3)$ such that
\[
\| u \|_2^2 = \mu', \quad \| v\|_2^2 = \mu-\mu', 
\]
\begin{equation} \label{eq:3.4}
E(u) \le c\left(\mu^{\prime} \right) +\frac{\varepsilon}{2} 
\quad \hbox{and} \quad 
E_{\infty}(v) \le c_{\infty} \left(\mu-\mu^{\prime}\right) +\frac{\varepsilon}{2}.
\end{equation}
For $k \in \R^3$, we put $v_k(x) := v(x-k)$.
Then for sufficiently large $|k|$, it follows that 
\begin{equation} \label{eq:3.5}
\operatorname{supp} u \cap \operatorname{supp} v_k = \emptyset
\quad \text{and} \quad \| u + v_k \|_2^2 = \mu.
\end{equation}

Now from \ef{eq:3.5}, one finds that 
\begin{align} \label{eq:3.6}
S_1(u+v_k) 
& =\frac{1}{8 \pi} \int_{\mathbb{R}^3} \frac{|u(y)+v_k(y)|^2}{|x-y|} \,d y 
=\frac{1}{8 \pi} \int_{\mathbb{R}^3} \frac{|u(y)|^2}{|x-y|} \,d y
+\frac{1}{8 \pi} \int_{\mathbb{R}^3} \frac{|v_k(y)|^2}{|x-y|} \,d y \notag \\
&=S_1(u)+S_1(v_k).
\end{align}
Then it holds that
\[
\begin{aligned}
A_1(u+v_k) 
& =\frac{1}{4} \int_{\mathbb{R}^3} 
S_1(u+v_k) |u+v_k|^2 \,d x 
 =\frac{1}{4} \int_{\mathbb{R}^3} \big( S_1(u)+S_1(v_k) \big)
\left( | u |^2+ | v_k |^2\right) \,d x \\
&=\frac{1}{4} \int_{\mathbb{R}^3} S_1(u) | u |^2 \,d x
+\frac{1}{4} \int_{\mathbb{R}^3} S_1(v_k) | v_k |^2 \,d x
+\frac{1}{2} \intR S_1(u)| v_k |^2 \,dx \\
&= A_1(u)+A_1(v_k) + R(k),
\end{aligned}
\]
where 
\[
\begin{aligned}
R(k) &:= \frac{1}{2} \intR S_1(u)| v_k |^2 \,dx
= \frac{1}{16 \pi} \int_{\mathbb{R}^3} \intR 
\frac{|u(y)|^2 |v_k(x)|^2}{|x-y|} \,d x\,d y \\
&=\frac{1}{16 \pi} \int_{{\rm supp}\,u} \int_{{\rm supp}\,v+k } 
\frac{|u(y)|^2 |v(x-k)|^2}{|x-y|} \,d x\,d y.
\end{aligned}
\]
If $x \in {\rm supp}\,v+k$ and $y \in {\rm supp}\,u$, it follows that 
$\dis |x-y| \ge |k|- \sup_{t \in {\rm supp}\, v, s \in {\rm supp}\, u} |t-s|$,
provided that $|k|$ is sufficiently large. 
Then one has
\[
R(k) \le
\frac{1}{16\pi \{ |k|- \sup_{t \in {\rm supp}\, v, s \in {\rm supp}\, u} |t-s| \} }
\| u \|_{2}^2 \| v \|_{2}^2 
\rightarrow 0 \text { as } |k| \rightarrow \infty.
\]
Moreover from \ef{eq:1.8} and \ef{eq:3.6}, we get
\begin{align} \label{eq:3.7}
A_2(u+v_k) &= - \frac{1}{4} \intR S_1(u+v_k) \rho(x) \,dx 
= - \frac{1}{4} \intR S_1(u) \rho(x) \,dx - \frac{1}{4} \intR S_1(v_k) \rho(x) \,dx \notag \\
&= A_2(u) + A_2(v_k) \le A_2(u).
\end{align}
Since $\| u + v_k \|_2^2 = \mu$, we have from \ef{eq:3.4} and \ef{eq:3.7} that
\[
c(\mu) \le E(u+v_k) = E(u) +E_{\infty}(v_k)+R(k) 
\le c\left(\mu^{\prime}\right) +c_{\infty}\left(\mu-\mu^{\prime}\right) 
+\varepsilon+R(k).
\]
Here we used the fact that $E_{\infty}(u)$ is translation invariant.
Thus one has
\[
c(\mu) \le \limsup_{|k| \to \infty} 
\left\{ c\left(\mu^{\prime}\right) 
+ c_{\infty}\left(\mu-\mu^{\prime}\right) +\ep+R(k)\right\} 
= c \left(\mu^{\prime}\right) +c_{\infty}\left(\mu-\mu^{\prime}\right)+\varepsilon.
\]
Passing to a limit $\ep \to 0+$ and using the fact that $c_{\infty}(\mu- \mu') \le 0$, 
we obtain
\[
c(\mu) \le c\left(\mu^{\prime}\right) +c_{\infty}\left(\mu -\mu^{\prime}\right) 
\le c\left( \mu^{\prime} \right).
\]

(ii) Let $\mu>0$ be fixed. 
First we infer that 
\[
\lim_{h \rightarrow 0+} c(\mu-h)=c(\mu).
\]
If $c(\mu)=0$, (i) implies that 
\[
0 = c(\mu) \le c( \mu -h) \le 0
\]
and hence the claim follows trivially.
Thus we may assume that $c(\mu)<0$.

For any $u \in B(\mu)$ and $h \in (0, \mu)$, we put
$u_h(x) := \sqrt{1- \frac{h}{\mu}} u(x)$.
Then it holds that 
\[
\| u_h \|_2^2 = \left( 1- \frac{h}{\mu} \right) \| u \|_2^2
= \mu- h \quad \text{and} \quad u_h \to u \ \text{in} \ H^1(\R^3) \textnormal{ as } h \to 0.
\]
Thus by Lemma \ref{lem:2.2}, one has
\[
\limsup_{h \to 0+} c(\mu -h) \le \lim_{h \to 0+} E(u_h) = E(u).
\]
Since $u \in B(\mu)$ is arbitrary, using (i) again, we find that
\[
c(\mu) \le \liminf_{h \to 0+} c( \mu -h) 
\le \limsup_{h \to 0+} c( \mu -h) \le c(\mu).
\]

Next we claim that
\[
\lim_{h \to 0+} c(\mu +h)= c(\mu).
\]
From (i), we know that $c(\mu+h) \le c(\mu)$ for any $h>0$.
Moreover it suffices to consider the case $h= \frac{1}{n}$ for $n \in \N$.
Let us choose $u_n \in B \left( \mu + \frac{1}{n} \right)$ so that
$E(u_n) \le c \left( \mu + \frac{1}{n} \right) + \frac{1}{n}$.
Then by the proof of Lemma \ref{lem:3.1}, 
$\{ u_n \}$ is bounded in $H^1(\R^3,\C)$.
Putting $v_n(x) := \sqrt{\frac{n \mu}{n \mu +1}} u_n(x)$, 
we have $v_n \in B(\mu)$ and
\[
\| v_n - u_n \|_{H^1(\R^3)}
= \left( 1- \sqrt{\frac{n\mu }{n\mu +1}} \, \right) \| u_n \|_{H^1(\R^3)} \rightarrow 0
\quad \text{as} \ n \to \infty.
\]
This implies that $E(v_n) = E(u_n) +o(1)$ and hence
\[
c(\mu) \le \liminf_{n \rightarrow \infty} E(v_n)
= \liminf_{n \to \infty} E(u_n) \le \liminf_{n \rightarrow \infty} 
c \left( \mu +\frac{1}{n}\right).
\]
Thus we obtain
\[
c(\mu) \le \liminf_{h \to 0+} c(\mu+h) \le \limsup_{h \to 0+} c( \mu+h) \le c(\mu).
\]

(iii) It is known that $c_{\infty}(\lambda \mu) \le \lambda c_{\infty}(\mu)$; 
See e.g. \cite{CW, JL}.
Thus from \ef{eq:3.2}, we obtain
\[
c( \lambda \mu) \le c_{\infty}(\lambda \mu) \le  \lambda c_{\infty}(\mu), 
\quad \text{for any $\mu>0$ and $\lambda >1$},
\]
which completes the proof.
\end{proof} 

Under these preparations, we are able to prove the strict sub-additivity
when $2<p< \frac{7}{3}$. 
The next lemma states that a smaller mass cannot escape to infinity
if the total mass is large.

\begin{lemma} \label{lem:3.4}
Suppose that $2 < p < \frac{7}{3}$ and let $\mu > 2\cdot 2^{\frac{1}{2p-4}} \mu^*$ be given. 
There exists $\rho_0= \rho_0(e, \mu)>0$ such that 
if $\|\rho\|_{\frac{6}{5}}+\|x \cdot \nabla \rho\|_{\frac{6}{5}} \le \rho_0$,
the following properties hold.

\begin{enumerate}
\item[\rm(i)] $c( \lambda s) < \lambda c(s)$ 
for all $\lambda >1$ and $\frac{\mu}{2\cdot 2^{\frac{1}{2p-4}}} \le s \le \mu$. 

\smallskip
\item[\rm(ii)] 
$c(\mu) < c\left(\mu^{\prime}\right) + c_{\infty} \left(\mu-\mu^{\prime}\right)$
for all $\frac{\mu}{2\cdot 2^{\frac{1}{2p-4}}} \le \mu' < \mu$
and $c(\mu) < c\left(\mu - \mu^{\prime}\right) + c_{\infty} (\mu^{\prime})$
for all $0< \mu' < \left( 1- \frac{1}{2\cdot 2^{\frac{1}{2p-4}}} \right) \mu $.

\end{enumerate}

\end{lemma}

\begin{proof} 
(i) First by Lemma \ref{lem:3.3} (i), \ef{eq:1.6}, 
\ef{eq:3.3} and from the fact $\frac{\mu}{2\cdot 2^{\frac{1}{2p-4}}} > \mu^*$, we have
\[
c(s) \le c\left( \frac{\mu}{2\cdot 2^{\frac{1}{2p-4}}} \right) 
< \frac{1}{2} c_{\infty}\left( \frac{\mu}{2\cdot 2^{\frac{1}{2p-4}}} \right)< 0,
\]
provided that $\| \rho \|_{\frac{6}{5}}$ is small.
For any $\ep \in \left( 0, - \frac{1}{4} c_{\infty} \Big( \frac{\mu}{2\cdot 2^{\frac{1}{2p-4}}} \Big) \right)$, 
there exists $u_{\ep} \in B(s)$ such that
\begin{equation} \label{eq:3.8}
E(u_{\ep}) \le c(s) + \ep \le \frac{1}{4} c_{\infty}\left( \frac{\mu}{2\cdot 2^{\frac{1}{2p-4}}} \right) < 0.
\end{equation}
Next for $a$, $b \in \R$ satisfying $2a-3b=1$, 
let us consider a scaling $(u_{\ep})_{\lambda}
:= \lambda^a u_{\ep} \left(\lambda^b x\right)$ so that the following property holds.
\begin{equation} \label{eq:3.9}
\left\| (u_{\ep})_{\lambda} \right\|_2^2
=\lambda^{2 a-3 b}\| u_{\ep} \|_2^2= \lambda s.
\end{equation}
Then from \ef{eq:2.5} and \ef{eq:2.6}, it follows that
\[
\begin{aligned}
E \big( (u_{\ep})_{\lambda} \big)
&= \frac{\lambda^{1+2b}}{2} \|\nabla u_{\ep} \|_2^2
-\frac{\lambda^{\frac{p+1+3(p-1) b}{2}}}{p+1} \| u_{\ep} \|_{p+1}^{p+1} \\
&\quad +e^2 \lambda^{2+b} A_1(u_{\ep} )
-\frac{\lambda^{1-2b}}{2} e^2 \intR S_1(u_{\ep}) \rho (\lambda^{-b} x) \,dx.
\end{aligned}
\]
We define a function $f(\lambda):[1,\infty) \to \R$ by
\[
\begin{aligned}
f(\lambda) &:= \frac{1}{\lambda^{1+2b}} E \big( (u_{\ep})_{\lambda} \big)- E(u_{\ep}) \\
&= \frac{1}{p+1} 
\left( 1- \lambda^{\frac{p-1+(3p-7)b}{2}} \right) \| u_{\ep} \|_{p+1}^{p+1}
-e^2 (1-\lambda^{1-b}) A_1(u_{\ep} ) \\
&\quad 
-2e^2 A_2(u_{\ep})
-\frac{\lambda^{-4b} }{2} e^2 \intR S_1(u_{\ep}) \rho (\lambda^{-b} x) \,d x.
\end{aligned}
\]
Note that $f(1)=0$.

We claim that there exists $\rho_0= \rho_0(e, \mu)>0$ such that 
if $\|\rho\|_{\frac{6}{5}}+\|x \cdot \nabla \rho\|_{\frac{6}{5}} \le \rho_0$,
\begin{equation} \label{eq:3.10}
f^{\prime}(\lambda) < 0 \quad \text{for all} \ \lambda >1 
\ \text{and} \ s \in \left[ \frac{\mu}{2\cdot 2^{\frac{1}{2p-4}}}, \mu \right].
\end{equation}
For this purpose, one computes
\[
\begin{aligned}
f'(\lambda)
&=-\frac{p-1+(3 p-7) b}{2(p+1)} \lambda^{\frac{p-3+(3p-7) b}{2}} 
\| u_{\ep} \|_{p+1}^{p+1} 
+(1-b) \lambda^{-b} e^2 A_1(u_{\ep}) \\
&\quad +2 b \lambda^{-1-4b} e^2 \intR S_1(u_{\ep}) \rho (\lambda^{-b} x) \,d x
+\frac{b}{2} \lambda^{-1-5 b} e^2 
\intR S_1(u_{\ep}) x \cdot \nabla \rho (\lambda^{-b}x ) \,dx.
\end{aligned}
\]
By the definition of $E$, we also have
\[
-\frac{1}{p+1}\| u_{\ep} \|_{p+1}^{p+1}
=E(u_{\ep} )-\frac{1}{2} \| \nabla u_{\ep} \|_2^2-e^2 A_1(u_{\ep})-2 e^2 A_2(u_{\ep}),
\]
and hence
\begin{align*} 
f^{\prime}(\lambda)
&= \frac{p-1+(3p-7)b}{2} \lambda^{\frac{p-3+(3p-7)b}{2}} E(u_{\ep}) 
- \frac{p-1+(3p-7)b}{4} \lambda^{\frac{p-3+(3p-7)b}{2}} \| \nabla u_{\ep} \|_2^2 \\
&\quad - e^2 \left( \frac{p-1+(3p-7)b}{2} \lambda^{\frac{p-3+(3p-7)b}{2}}
- (1-b) \lambda^{-b} \right) A_1(u_{\ep}) \\
&\quad - e^2 \big( p-1 +(3p-7)b \big) \lambda^{\frac{p-3+(3p-7)b}{2}} A_2(u_{\ep}) \\
&\quad +2 b \lambda^{-1-4b} e^2 \intR S_1(u_{\ep}) \rho (\lambda^{-b} x) \,d x
+\frac{b}{2} \lambda^{-1-5 b} e^2 
\intR S_1(u_{\ep}) x \cdot \nabla \rho (\lambda^{-b}x ) \,dx.
\end{align*}
Now since $2<p<\frac{7}{3}$, we can take $b>0$ so that 
\begin{equation} \label{eq:3.11}
1 \le b < \frac{p-1}{7-3 p}.
\end{equation}
Then one has 
\[
\frac{p-1+(3p-7)b}{2} \lambda^{\frac{p-3+(3p-7)b}{2}}
- (1-b) \lambda^{-b} \ge 0 \quad \text{for} \ \lambda > 1.
\]
Thus from \ef{eq:2.3}, one finds that
\[
\begin{aligned}
f^{\prime}(\lambda)
&\le \frac{p-1+(3p-7)b}{2} \lambda^{\frac{p-3+(3p-7)b}{2}} E(u_{\ep}) 
- \frac{p-1+(3p-7)b}{4} \lambda^{\frac{p-3+(3p-7)b}{2}} \| \nabla u_{\ep} \|_2^2 \\
&\quad - e^2 \big( p-1 +(3p-7)b \big) \lambda^{\frac{p-3+(3p-7)b}{2}} A_2(u_{\ep}) \\
&\quad +2 b \lambda^{-1-4b} e^2 \intR S_1(u_{\ep}) \rho (\lambda^{-b} x) \,d x
+\frac{b}{2} \lambda^{-1-5 b} e^2 
\intR S_1(u_{\ep}) x \cdot \nabla \rho (\lambda^{-b}x ) \,dx.
\end{aligned}
\]
Moreover by Lemma \ref{lem:2.1}, we obtain
\[
\begin{aligned}
f^{\prime}(\lambda) 
&\le \frac{p-1+(3p-7)b}{2} \lambda^{\frac{p-3+(3p-7)b}{2}}
\left( E(u_{\ep}) - \frac{1}{2} \| \nabla u_{\ep} \|_2^2
+ C_1 e^2 s^{\frac{3}{4}} \| \rho \|_{\frac{6}{5}} \| \nabla u_{\ep} \|_2^{\frac{1}{2}} \right) \\
&\quad + C_1 \lambda^{-1-4b} e^2 
\left( \| \rho(\lambda^{-b}x) \|_{\frac{6}{5}}
+ \| (\lambda^{-b}x) \cdot \nabla \rho (\lambda^{-b}x) \|_{\frac{6}{5}} \right)
\| S_1(u_{\ep}) \|_6 \\
&\le \frac{p-1+(3p-7)b}{2} \lambda^{\frac{p-3+(3p-7)b}{2}}
\left( E(u_{\ep}) - \frac{1}{2} \| \nabla u_{\ep} \|_2^2
+ C_1 e^2 s^{\frac{3}{4}} \| \rho \|_{\frac{6}{5}} \| \nabla u_{\ep} \|_2^{\frac{1}{2}} \right) \\
&\quad + C_1 \lambda^{-1-\frac{3b}{2}} e^2 s^{\frac{3}{4}} 
\left( \| \rho \|_{\frac{6}{5}} + \| x  \cdot \nabla \rho \|_{\frac{6}{5}} \right)
\| \nabla u_{\ep} \|_2^{\frac{1}{2}},
\end{aligned}
\]
where $C_1$ is a positive constant independent of $e$, $s$, $\rho$,
$\lambda$ and $u_{\ep}$.
By the Young inequality, the fact $s \le \mu$ and from \ef{eq:3.8}, it follows that
\[
\begin{aligned}
E(u_{\ep}) - \frac{1}{2} \| \nabla u_{\ep} \|_2^2
+ C_1 e^2 s^{\frac{3}{4}} \| \rho \|_{\frac{6}{5}} \| \nabla u_{\ep} \|_2^{\frac{1}{2}}
&\le \frac{1}{4} c_{\infty} \left( \frac{\mu}{2\cdot 2^{\frac{1}{2p-4}}} \right)
- \frac{1}{4} \| \nabla u_{\ep} \|_2^2 
+ C_2 e^{\frac{8}{3}} s \| \rho \|_{\frac{6}{5}}^{\frac{4}{3}} \\
&< \frac{1}{4} c_{\infty} \left( \frac{\mu}{2\cdot 2^{\frac{1}{2p-4}}} \right)
+ C_2 e^{\frac{8}{3}} \mu \| \rho \|_{\frac{6}{5}}^{\frac{4}{3}}
< 0
\end{aligned}
\]
for some $C_2>0$, provided that $\| \rho \|_{\frac{6}{5}}$ is small .
Since $\lambda >1$, we deduce that
\[
\begin{aligned}
\frac{f'(\lambda)}{\lambda^{-1-\frac{3b}{2}}}
&\le \frac{p-1+(3p-7)b}{2} 
\lambda^{\frac{p-1+(3p-4)b}{2}} 
\left( \frac{1}{4} c_{\infty} \left( \frac{\mu}{2\cdot 2^{\frac{1}{2p-4}}} \right)
- \frac{1}{4} \| \nabla u_{\ep} \|_2^2 
+ C_2 e^{\frac{8}{3}} \mu \| \rho \|_{\frac{6}{5}}^{\frac{4}{3}} \right) \\
&\quad + C_1 e^2 \mu^{\frac{3}{4}} 
\left( \| \rho \|_{\frac{6}{5}} + \| x  \cdot \nabla \rho \|_{\frac{6}{5}} \right)
\| \nabla u_{\ep} \|_2^{\frac{1}{2}} \\
&\le \frac{p-1+(3p-7)b}{2} 
\left( \frac{1}{4} c_{\infty} \left( \frac{\mu}{2\cdot 2^{\frac{1}{2p-4}}} \right)
- \frac{1}{4} \| \nabla u_{\ep} \|_2^2 
+ C_2 e^{\frac{8}{3}} \mu \| \rho \|_{\frac{6}{5}}^{\frac{4}{3}} \right) \\
&\quad + C_1 e^2 \mu^{\frac{3}{4}} 
\left( \| \rho \|_{\frac{6}{5}} + \| x  \cdot \nabla \rho \|_{\frac{6}{5}} \right)
\| \nabla u_{\ep} \|_2^{\frac{1}{2}} \\
&\le \frac{p-1+(3p-7)b}{8} c_{\infty} \left( \frac{\mu}{2\cdot 2^{\frac{1}{2p-4}}} \right)
+ C_3 e^{\frac{8}{3}} \mu 
\left( \| \rho \|_{\frac{6}{5}} + \| x  \cdot \nabla \rho \|_{\frac{6}{5}} \right)^{\frac{4}{3}}
\end{aligned}
\]
for some $C_3>0$. 
Thus there exists $\rho_0>0$ which depends only on $e$ and $\mu$ such that
\[
\|\rho\|_{\frac{6}{5}}+\|x \cdot \nabla \rho\|_{\frac{6}{5}} \le \rho_0
\ \Rightarrow \ f^{\prime}(\lambda) < 0
\quad \text{for all} \ \lambda >1 \ \text{and} \ s \in \left[ \frac{\mu}{2\cdot 2^{\frac{1}{2p-4}}}, \mu \right].
\]
This ends the proof of \ef{eq:3.10}. 

Now from \ef{eq:3.10} and $f(1)=0$, 
it follows that $f( \lambda) < 0$ for all $\lambda >1$.
Recalling that 
$E\big( (u_{\ep})_\lambda \big)-\lambda^{1+2b} E(u_{\ep})=\lambda^{1+2b} f(\lambda)$
and from \ef{eq:3.9}, one finds that 
\[
c(\lambda s) \leq E\left((u_{\ep})_{\lambda} \right)
<\lambda^{1+2 b} E(u_{\ep}) \leq \lambda^{1+2 b}(c(s)+\varepsilon).
\]
Taking a limit $\varepsilon \rightarrow 0$, we get
\[
c(\lambda s) \leq \lambda^{1+2 b} c(s).
\]
Since $\lambda>1$, $b>0$ and from \ef{eq:3.3}, it holds that
\[
c(\lambda s) < \lambda c(s) \quad \hbox{for all } \ \lambda >1
\ \text{and} \ s \in \left[ \frac{\mu}{2\cdot 2^{\frac{1}{2p-4}}}, \mu \right].
\]

(ii) First we suppose that $\frac{\mu}{2} < \mu' < \mu$.
Then it follows that 
\[
0<\mu-\mu'<\mu', \ \frac{\mu}{\mu'} >1, \ \frac{\mu'}{\mu-\mu'} > 1 \quad \text{and} \quad 
\mu' > \frac{\mu}{2} > \frac{\mu}{2 \cdot 2^{\frac{1}{2p-4}}}.
\]
Applying (i) with $s=\mu'$ and $\lambda = \frac{\mu}{\mu'}$, one has
\[
c \left( \frac{\mu}{\mu'} \mu' \right) < \frac{\mu}{\mu'} c( \mu').
\]
Thus using Lemma \ref{lem:3.3} (iii), we obtain
\[
\begin{aligned}
c(\mu) &= c \left( \frac{\mu}{\mu'} \mu' \right) 
< \frac{\mu}{\mu'} c(\mu') 
= c(\mu') + \frac{\mu-\mu'}{\mu'} c \left( \frac{\mu'}{\mu-\mu'}(\mu-\mu') \right) \\
&\le c(\mu') + \frac{\mu-\mu'}{\mu'} \cdot \frac{\mu'}{\mu-\mu'}
c_{\infty} (\mu-\mu')
= c(\mu') + c_{\infty}(\mu-\mu'),
\end{aligned}
\]
from which we conclude.
In the case $\frac{\mu}{2}= \mu'$, we find by $\mu' = \mu-\mu'$ and \ef{eq:3.2} that
\[
c(\mu) < 2 c(\mu') = c(\mu') + c(\mu-\mu') 
\le c(\mu') + c_{\infty}(\mu-\mu').
\]
Next we assume that $\frac{\mu}{2\cdot 2^{\frac{1}{2p-4}}} \le \mu' < \frac{\mu}{2}$.
Then it holds that 
\[
0< \mu' < \mu- \mu', \ \frac{\mu}{\mu-\mu'}>1, \ \frac{\mu-\mu'}{\mu'}>1 \quad \text{and} \quad
\mu-\mu' > \frac{\mu}{2} > \frac{\mu}{2\cdot 2^{\frac{1}{2p-4}}}.  
\]
First we use (i) with $s= \mu-\mu'$ and $\lambda = \frac{\mu}{\mu-\mu'}$ to obtain
\[
c(\mu) = c \left( \frac{\mu}{\mu-\mu'} (\mu-\mu') \right) 
< \frac{\mu}{\mu-\mu'} c(\mu-\mu') 
= \frac{\mu'}{\mu-\mu'} 
c \left( \frac{\mu-\mu'}{\mu'}\mu' \right) + c(\mu-\mu').
\]
Next applying (i) with $s=\mu'$ and $\lambda= \frac{\mu-\mu'}{\mu'}$, one gets
\[
c(\mu) < c(\mu') + c(\mu-\mu').
\]
Finally by \ef{eq:3.2}, we deduce that
\[
c(\mu) < c(\mu') +c_{\infty}(\mu-\mu').
\]

When $0< \mu' \le \left( 1- \frac{1}{2 \cdot 2^{\frac{1}{2p-4}}} \right) \mu$, it follows that
$ \frac{\mu}{2\cdot 2^{\frac{1}{2p-4}}} < \mu - \mu' < \mu$.
In the case $\frac{\mu}{2} \le \mu - \mu' < \mu$, we have
\[
\frac{\mu}{\mu-\mu'}>1, \ \frac{\mu-\mu'}{\mu'} \ge 1 \quad \text{and} \quad
\mu-\mu' \ge \frac{\mu}{2 \cdot 2^{\frac{1}{2p-4}}}.
\]
Thus from (i), Lemma \ref{lem:3.3} (iii) and \ef{eq:3.2}, one gets
\[
\begin{aligned}
c(\mu) &= c \left( \frac{\mu}{\mu-\mu'} (\mu-\mu') \right) 
< \frac{\mu}{\mu-\mu'} c(\mu-\mu') 
= c(\mu-\mu') + \frac{\mu'}{\mu-\mu'} 
c \left( \frac{\mu-\mu'}{\mu'}\mu' \right) \\
&\le c(\mu-\mu') + c_{\infty}(\mu').
\end{aligned}
\]
If $\frac{\mu}{2\cdot 2^{\frac{1}{2p-4}}} \le \mu- \mu' < \frac{\mu}{2}$, it holds that
\[
\frac{\mu}{\mu'}>1, \ \frac{\mu'}{\mu-\mu'}>1 \quad \text{and} \quad
\mu' \ge \frac{\mu}{2 \cdot 2^{\frac{1}{2p-4}}}.
\]
Thus using (i) twice and \ef{eq:3.2}, we deduce that
\[
\begin{aligned}
c(\mu) &= c \left( \frac{\mu}{\mu'} \mu' \right) 
< \frac{\mu}{\mu'} c(\mu') 
= \frac{\mu-\mu'}{\mu'} c \left( \frac{\mu'}{\mu-\mu'}(\mu-\mu') \right) + c(\mu') \\
&< c(\mu-\mu') + c_{\infty} (\mu'),
\end{aligned}
\]
which completes the proof. 
\end{proof} 

\begin{remark} \label{rem:3.5}
In the case $1<p \le 2$, 
it holds that $\frac{p-1}{7-3p} \le 1$ 
and hence we cannot choose $b>0$ so that \ef{eq:3.11} holds. 
Note that the condition $b \ge 1$ was used to remove
$A_1(u)$ which is independent of $\rho$.
We also mention that the choice $b=1$ corresponds to the scaling
$u_{\lambda}(x)= \lambda^2 u( \lambda x)$. 
\end{remark}

Next we show that a larger mass cannot escape to infinity if the total mass is large.

\begin{lemma} \label{lem:3-ex}
Suppose that $2< p< \frac{7}{3}$ and let $\mu > 2\cdot 2^{\frac{1}{2p-4}} \mu^*$ be given.
We also assume \ef{ASS}.

\begin{enumerate}
\item[\rm(i)] If $0<s< \mu$ and $c(s)<0$, then $c(\lambda s) < \lambda c(s)$ for all $\lambda \ge 2^{\frac{1}{2p-4}}$.

\item[\rm(ii)] If $0< \mu' < \frac{\mu}{2 \cdot 2^{\frac{1}{2p-4}}}$ and $c(\mu')<0$, then
$c(\mu) < c(\mu') + c_{\infty}(\mu- \mu')$.
If $\Big( 1- \frac{1}{2\cdot 2^{\frac{1}{2p-4}}} \Big) \mu < \mu' < \mu$ and $c(\mu-\mu')<0$, then
$c(\mu) < c(\mu - \mu') + c_{\infty}(\mu')$.

\end{enumerate}

\end{lemma}

\begin{proof}
(i) First we claim that if $E(u) \le 0$, it holds that
\begin{equation} \label{eq:ex-1}
-2e^2 A_2(u) \le \frac{1}{p+1} \| u \|_{p+1}^{p+1}.
\end{equation}
Indeed from \ef{ASS}, \ef{eq:2.3} and \ef{eq:2.4}, we have
\begin{align*}
0 \ge E(u) &= \frac{1}{2} \| \nabla u \|_2^2 - \frac{1}{p+1} \| u \|_{p+1}^{p+1} + e^2 A_1(u) + 2e^2 A_2(u) \\
&\ge \frac{1}{2} \IT \left( | \nabla u|^2 + 2e^2 S_2(x) |u|^2 \right) dx - \frac{1}{p+1} \| u \|_{p+1}^{p+1} - 2e^2 A_2(u) \\
&\ge - \frac{1}{p+1} \| u \|_{p+1}^{p+1} - 2e^2 A_2(u).
\end{align*}

Now for any $\ep>0$, we take $u_{\ep} \in B(s)$ so that 
\[
E(u_{\ep}) \le c(s) + \ep \le 0
\]
and put $(u_{\ep})_{\lambda}(x) = \lambda^2 u_{\ep} (\lambda x)$.
As in the proof of Lemma \ref{lem:3.4}, one finds that $\| (u_{\ep} )_{\lambda} \|_2^2 = \lambda s$ and
\begin{align*}
f(\lambda) &:= \frac{1}{\lambda^3} E \big( (u_{\ep})_{\lambda} \big) - E(u_{\ep}) \\
&= \frac{1}{p+1} \left( 1- \lambda^{2p-4} \right) \| u_{\ep} \|_{p+1}^{p+1}
-2e^2 A_2(u_{\ep}) - \frac{\lambda^{-4}}{2} e^2 
\IT S_1(u_{\ep}) \rho( \lambda^{-1} x) \,dx.
\end{align*}
Since $E(u_{\ep}) \le 0$, we are able to use \ef{eq:ex-1} to obtain
\[
f(\lambda) \le \frac{1}{p+1} \left( 2 - \lambda^{2p-4} \right) \| u_{\ep} \|_{p+1}^{p+1}
\le 0 \quad \text{for all} \ \lambda \ge 2^{\frac{1}{2p-4}},
\]
which yields that
\[
c(\lambda s) \le \lambda^3 E(u_{\ep}) \le \lambda^3 \big( c(s) + \ep \big).
\]
Taking a limit $\ep \to 0$ and recalling the fact $c(s)<0$, we get
\[
c(\lambda s) \le \lambda^3 c(s) < \lambda c(s) \quad \text{for all} \ \lambda \ge 2^{\frac{1}{2p-4}}.
\]

(ii) Suppose that $0< \mu' < \frac{\mu}{2\cdot 2^{\frac{1}{2p-4}}}$ and $c(\mu')<0$.
Then it follows that $0< \mu' < \mu-\mu'$,  
\[
\mu > 2 \cdot 2^{\frac{1}{2p-4}} \mu' > \left( 1+ 2^{\frac{1}{2p-4}} \right) \mu', \ 
\frac{\mu-\mu'}{\mu'} > 2^{\frac{1}{2p-4}} \quad \text{and} \quad
\mu-\mu' > \left( 1- \frac{1}{2 \cdot 2^{\frac{1}{2p-4}}} \right) \mu > \frac{\mu}{2 \cdot 2^{\frac{1}{2p-4}}}.
\]
By using Lemma \ref{lem:3.4} (i) with $s= \mu-\mu'$ and $\lambda = \frac{\mu}{\mu-\mu'}>1$, one has
\[
c(\mu) = c \left( \frac{\mu}{\mu-\mu'} (\mu-\mu') \right) 
< \frac{\mu}{\mu-\mu'} c(\mu-\mu') 
= \frac{\mu'}{\mu-\mu'} 
c \left( \frac{\mu-\mu'}{\mu'}\mu' \right) + c(\mu-\mu').
\]
Next we apply (i) with $s=\mu'$ and $\lambda= \frac{\mu-\mu'}{\mu'}$ to obtain
\[
c(\mu) < c(\mu') + c(\mu- \mu').
\]
Finally from \ef{eq:3.2}, we deduce that
\[
c(\mu) < c(\mu') + c_{\infty} (\mu - \mu').
\]

When $\left( 1- \frac{1}{2 \cdot 2^{\frac{1}{2p-4}}} \right) \mu < \mu' < \mu$ and $c(\mu-\mu')<0$, 
it holds that 
\[
0<\mu-\mu'< \mu', \ \frac{\mu'}{\mu-\mu'} > 2^{\frac{1}{2p-4}} \quad \text{and} \quad
\mu' > \left( 1- \frac{1}{2\cdot 2^{\frac{1}{2p-4}}} \right) \mu > \frac{\mu}{2 \cdot 2^{\frac{1}{2p-4}}}.
\]
Applying Lemma \ref{lem:3.4} (i) with $s=\mu'$ and $\lambda = \frac{\mu}{\mu'}>1$, one finds that 
\[
c(\mu) = c \left( \frac{\mu}{\mu'} \mu' \right) 
< \frac{\mu}{\mu'} c(\mu') 
= \frac{\mu-\mu'}{\mu'} c \left( \frac{\mu'}{\mu-\mu'}(\mu-\mu') \right) + c(\mu').
\]
Next we use (i) with $s= \mu-\mu'$ and $\lambda = \frac{\mu'}{\mu-\mu'}$ to conclude that
\[
c(\mu) < c(\mu-\mu') + c(\mu') \le c(\mu-\mu') + c_{\infty}(\mu').
\]
This completes the proof.
\end{proof} 

The next lemma deals with the compactness of any minimizing sequence for 
\ef{eq:3.1}.

\begin{lemma} \label{lem:3.6}
Suppose that $2<p<\frac{7}{3}$.
For any $\mu > 2\cdot 2^{\frac{1}{2p-4}} \mu^*$, there exists 
$\rho_0= \rho_0(e, \mu)>0$ such that 
if $\|\rho\|_{\frac{6}{5}}+\|x \cdot \nabla \rho\|_{\frac{6}{5}} \le \rho_0$,
the following properties hold.

Let $\{ u_j \} \subset H^1(\R^3,\C)$ be a sequence satisfying 
$\| u_j \|_{2}^2 \to \mu$ and $E(u_j) \to c(\mu)$.
Then there exist a subsequence of $\{ u_j \}$ which is still denoted by the same,
a sequence $\{ y_j \} \subset \R^3$ and $u_{\mu} \in H^1(\R^3,\C)$ such that
$u_j(\cdot-y_j) \to u_{\mu}$ in $H^1(\R^3,\C)$ and $E(u_{\mu})=c(\mu)$.
\end{lemma} 

Lemma \ref{lem:3.6} seems to be rather standard,
once we have established the strict sub-additivity.
However since \ef{eq:1.1} is non-autonomous, 
small and large masses play different roles,
causing that the proof of Lemma \ref{lem:3.6} is not straightforward.

\begin{proof}
First we observe by the proof of Lemma \ref{lem:3.1} 
that $\| u_j \|_{H^1}$ is bounded.
Moreover by replacing $u_j$ by $ \frac{ \sqrt{\mu}}{ \| u_j \|_{2}^2} u_j$,
we may assume that $\{ u_j \}$ is a minimizing sequence of $c(\mu)$.
It is also important to note that $c(\mu) < 0$ from \ef{eq:3.3}.

Now we apply the concentration compactness principle \cite[Lemma I.1, p. 115]{L1}
to the sequence $\rho_j(x)= | u_j(x)|^2$.
It is well-known that the behavior of the sequence $(\rho_j)_{j\in\N}$
is governed by the three possibilities:
Compactness, Vanishing and Dichotomy.
Our goal is to show that Compactness occurs.

If Vanishing occurs, there exists a subsequence of $\{ \rho_j \}$, 
still denoted by $\{ \rho_j \}$, such that
$$
\lim_{j \to \infty} \sup_{y\in\R^3} \int_{B_R(y)} \rho_j(x) \,dx=0
\quad \hbox{for all} \ R>0.$$
Here $B_R(y)$ describes a ball of radius $R$ with the center at $y \in \R^3$.
Then by \cite[Lemma I.1, P. 231]{L2}, it follows that 
$u_j \to 0$ in $L^q(\R^3)$ for any $q\in (2,6)$.
On the other hand since $\{ u_j \}$ is a minimizing sequence for $c(\mu)$, 
one has by Lemma \ref{lem:2.1} that
\begin{align*}
c(\mu)+o(1) &= E(u_j) 
= \frac{1}{2} \| \nabla u_j \|_{2}^2 -\frac{1}{p+1} \| u_j \|_{p+1}^{p+1}
+ e^2 A_1(u_j) +2e^2 A_2(u_j) \\
&\ge -\frac{1}{p+1} \| u_j \|_{p+1}^{p+1}
-C \| \rho \|_{\frac{6}{5}} \| u_j\|_{\frac{12}{5}}^2. 
\end{align*}
Passing a limit $j \to \infty$, we get $0>c(\mu) \ge 0$.
This is a contradiction, which rules out Vanishing.

Next we assume that Dichotomy occurs.
Then by a standard argument (see \cite[Section I.2]{L2} or 
\cite[Proposition 1.7.6, P. 23]{Ca}), 
there exist $\mu' \in (0,\mu)$, $R_1$, $R_2>0$, 
$\{ y_j \}$, $\{ z_j \} \subset \R^3$
and $\{ u_{j,1} \}$, $\{ u_{j,2} \} \subset 
H^1(\R^3,\C)$ so that
\[
\| u_{j,1} \|_{L^2}^2 \to \mu', \quad
\| u_{j,2} \|_{L^2}^2 \to \mu-\mu',
\]
\begin{equation} \label{eq:3.12}
{\rm supp}(u_{j,1}) \subset B_{R_1}(y_j), \ 
{\rm supp}(u_{j,2}) \subset B_{R_2}(z_j), \ 
|y_j-z_j| \to \infty, 
\end{equation}
\begin{equation} \label{eq:3.13}
\| u_j -u_{j,1} -u_{j,2} \|_{q} \to 0 \quad \hbox{for all} \ 2 \le q<6,
\end{equation}
\begin{equation} \label{eq:3.14}
\IT \big( | \nabla u_j|^2
-|\nabla u_{j,1}|^2 -| \nabla u_{j,2}|^2 \big) \,dx \ge o(1).
\end{equation}
Moreover replacing $u_{j,1}$, $u_{j,2}$ by $\frac{ \sqrt{\mu'}}{ \| u_{j,1} \|_{2}}u_{j,1}$, 
$\frac{ \sqrt{ \mu-\mu'}}{ \| u_{j,2} \|_{2}} u_{j,2}$ respectively, 
we may assume that $\| u_{j,1} \|_{2}^2 = \mu'$, $\| u_{j,2} \|_{2}^2=\mu-\mu'$
and \ef{eq:3.12}-\ef{eq:3.14} hold.
Now from \ef{eq:3.12}, one has
\begin{align*}
\IT \IT \frac{ | u_{j,1}(x) |^2 |u_{j,2}(y) |^2 }{|x-y|} \,dx \,dy
&= \int_{ {\rm supp}(u_{j,2})} \int_{ {\rm supp}(u_{j,1})} 
\frac{ |u_{j,1}(x)|^2 | u_{j,2}(y) |^2}{|x-y|} \,dx \,dy \\
& \le \frac{1}{|y_j-z_j| - R_1-R_2}
\| u_{j,1} \|_{2}^2 \| u_{j,2} \|_{2}^2 \to 0 
\quad \hbox{as} \ j \to \infty.
\end{align*}
Using \ef{eq:3.13} and arguing as in the proof of Lemma 2.2 in \cite{ZZ},
a direct computation yields that 
\begin{align*}
& A_1(u_j)-A_1(u_{j,1})-A_1(u_{j,2}) 
= \IT S_1(u_j)|u_j|^2-S_1(u_{j,1})|u_{j,1}|^2-S_1(u_{j,2})|u_{j,2}|^2 \,dx \\
&= \IT \Big\{ \big( S_1(u_j)|u_j|+S_1(u_{j,1})|u_{j,1}|+S_1(u_{j,2})|u_{j,2}| \big)
\big( |u_j|-|u_{j,1}|-|u_{j,2}| \big) \\
&\qquad \qquad 
+|u_j| \big( |u_{j,1}| +|u_{j,2}| \big) \big( S_1(u_j)-S_1(u_{j,1})-S_1(u_{j,2}) \big) \\
&\qquad \qquad + | u_{j,1} | |u_{j,2}| \big( S_1(u_{j,1})+S_1(u_{j,2}) \big)
+|u_j| \big( |u_{j,1}|S_1(u_{j,2}) +|u_{j,2}|S_1(u_{j,1}) \big) \Big\} \,dx \\
&\to 0 \quad (n\to \infty).
\end{align*}
Similarly, one finds that
\[
A_2(u_j)-A_2(u_{j,1})-A_2(u_{j,2}) \to 0.
\]
Moreover from \ef{eq:2.12} and \ef{eq:3.12}, it holds that
\begin{equation} \label{eq:3.15}
A_2(u_{j,2}) \to 0 \ \text{if $|y_j|$ is bounded} \quad \text{and} \quad
A_2(u_{j,1}) \to 0 \ \text{if} \ |y_j| \to \infty
\quad \text{as} \ j \to \infty. 
\end{equation} 

Now we suppose that $|y_j|$ is bounded.
Then we deduce from \ef{eq:3.13}, \ef{eq:3.14} and \ef{eq:3.15} that
\[
c(\mu) 
= E(u_j) +o(1) 
\ge E(u_{j,1}) + E_{\infty}(u_{j,2}) +o(1) 
\ge c(\mu') + c_{\infty}(\mu-\mu') +o(1).
\]
Taking $\dis \liminf_{j \to \infty}$ in both sides, one finds that 
\begin{equation} \label{eq:ex-2}
c(\mu) \ge c(\mu') + c_{\infty}(\mu-\mu').
\end{equation}
When $\frac{\mu}{2\cdot 2^{\frac{1}{2p-4}}} \le \mu' < \mu$, 
we have a contradiction by Lemma \ref{lem:3.4} (ii).
If $0< \mu' < \frac{\mu}{2\cdot 2^{\frac{1}{2p-4}}}$ and $c(\mu')=0$,
one has from \ef{eq:3.2} and \ef{eq:ex-2} that
\[
c_{\infty}(\mu) \ge c(\mu) \ge c_{\infty}(\mu-\mu').
\]
But since $\mu-\mu' > \mu^*$, it follows that $c_{\infty}(\mu-\mu') > c_{\infty}(\mu)$,
which leads to  a contradiction.
In the case $0< \mu' < \frac{\mu}{2\cdot 2^{\frac{1}{2p-4}}}$ and $c(\mu')<0$,
we are able to apply Lemma \ref{lem:3-ex} (ii) to arrive at a contradiction.

Next we assume that $|y_j| \to \infty$.
Then it follows from \ef{eq:3.13}, \ef{eq:3.14} and \ef{eq:3.15} that
\begin{equation} \label{eq:ex-3}
c(\mu) \ge c( \mu-\mu') + c_{\infty}(\mu').
\end{equation}
When $0< \mu' < \left( 1- \frac{1}{2\cdot 2^{\frac{1}{2p-4}}} \right) \mu$, we have a contradiction by Lemma \ref{lem:3.4} (ii).
If $\left( 1- \frac{1}{2\cdot 2^{\frac{1}{2p-4}}} \right) \mu \le \mu' < \mu$ and $c(\mu - \mu') =0$,
one finds that $c_{\infty}(\mu) \ge c(\mu) \ge c_{\infty} (\mu')$,
which is impossible because $c_{\infty}(\mu') > c_{\infty}(\mu)$.
In the case $\left( 1- \frac{1}{2\cdot 2^{\frac{1}{2p-4}}} \right) \mu \le \mu' < \mu$ and $c(\mu-\mu')<0$,
we can use Lemma \ref{lem:3-ex} (ii) to derive a contradiction.
Thus in all cases, we are able to reach a contradiction and hence Dichotomy does not occur. 

The only remaining possibility is Compactness, that is, 
there exists $\{ y_j \} \subset \R^3$ such that for all $\varepsilon>0$,
there exists $R_{\varepsilon}>0$ satisfying
\begin{equation} \label{eq:3.16}
\int_{ B_{R_{\varepsilon}} (y_j)} |u_j(x)|^2 \,dx \ge \mu - \varepsilon.
\end{equation}
Since $ \| u_j \|_{H^1}$ is bounded, there exists $u_{\mu} \in H^1(\R^3,\C)$ such that
up to a subsequence, \linebreak
$u_j(\cdot-y_j) \rightharpoonup u_{\mu}$ in $H^1(\R^3,\C)$.
Then from \ef{eq:3.16}, it follows that $ u_j(\cdot-y_j) \to u_{\mu}$ in $L^q(\R^3,\C)$ 
for any $2 \le q<6$.
By the weak lower semi-continuity of $ \| \nabla \cdot \|_{2}$
and by Lemma \ref{lem:2.2}, we get
\[
c(\mu) = \liminf_{j \to \infty} E \big( u_j( \cdot-y_j) \big)
\ge E(u_{\mu}) \ge c(\mu).
\]
This implies that $E(u_{\mu})=c(\mu)$ and 
$ \| \nabla u_j(\cdot-y_j) \|_{2} \to \| \nabla u_{\mu} \|_{2}$.
Thus we obtain $u_j(\cdot-y_j) \to u_{\mu}$ in $H^1(\R^3,\C)$ 
and hence the proof is complete.
\end{proof}

\begin{remark} \label{rem:3.7}
By the relation between $E$ and $\mathcal{E}$ in \ef{eq:2.4},
the relative compactness of minimizing sequences for \ef{eq:1.5}
also holds,
which will be applied to show the orbital stability later.  
\end{remark}

Now suppose that $u \in H^1(\R^3,\C)$ is a minimizer of \ef{eq:3.1}, 
that is, $E(u)=c(\mu)$ and $\|u\|_{2}^2=\mu$.
Up to a phase shift, we may assume that $u$ is real-valued.
Indeed by the well-known pointwise inequality $\big| \nabla |u| \big| \le | \nabla u|$,
one can see that $|u|$ is also a minimizer
and hence $u$ can be chosen to be real-valued.

Furthermore there exists a Lagrange multiplier $\omega=\omega(\mu) \in \R$ 
such that $u$ satisfies \ef{eq:1.1} with some constant $\omega(\mu)$.

\begin{lemma} \label{lem:3.8}
Suppose that $2 < p < \frac{7}{3}$ and 
let $\mu> 2\cdot 2^{\frac{1}{2p-4}} \mu^*$ be given. 
Then there exists 
$\rho_0= \rho_0(e, \mu)>0$ such that 
if $\|\rho\|_{\frac{6}{5}}+\|x \cdot \nabla \rho\|_{\frac{6}{5}} \le \rho_0$,
the Lagrange multiplier $\omega=\omega(\mu)$ is positive.
\end{lemma}

\begin{proof}
Let $u$ be a minimizer for $c(\mu)$. 
Then by Lemma \ref{lem:2.3} and from \ef{eq:3.3}, 
it follows that 
\[
\frac{(3 p-5) \omega \mu}{2} \ge
-\frac{5 p-7}{2} c_{\infty}(\mu) +2(p-2)\| \nabla u \|_2^2 + 8 e^2 A_2(u) -(3-p) A_3(u).
\]
Since $p>2$, one can choose $\varepsilon \in \big( 0, 2(p-2) \big)$.
Lemma \ref{lem:2.1} and the Young inequality yield that
\[
\begin{aligned}
\frac{(3 p-5) \omega \mu}{2} 
&\ge - \frac{5 p-7}{2} c_{\infty} (\mu) +2(p-2)\| \nabla u \|_2^2 
- C e^2 \mu^{\frac{3}{4}} \left( \| \rho\|_{\frac{6}{5}} 
+ \| x \cdot \nabla \rho \|_{\frac{6}{5}}  \right) \| \nabla u \|_2^{\frac{1}{2}} \\ 
&\ge - \frac{5 p-7}{2} c_{\infty} (\mu) + \big( 2(p-2)-\varepsilon \big) \| \nabla u \|_2^2 
- \frac{C}{\ep^{\frac{1}{3}}} e^{\frac{8}{3}} \mu 
\left( \| \rho\|_{\frac{6}{5}} + \| x \cdot \nabla \rho \|_{\frac{6}{5}} \right)^{\frac{4}{3}},
\end{aligned}
\]
where $C$ is a positive constant independent of $e$, $\mu$ and $\rho$.
Since $(5p-7) c_{\infty}(\mu)<0$ and $c_{\infty}(\mu)$ is independent of $\rho$, 
there exists $\rho_0= \rho_0(e, \mu)>0$ such that 
if $\|\rho\|_{\frac{6}{5}}+\|x \cdot \nabla \rho\|_{\frac{6}{5}} \le \rho_0$,
we have
\[
\frac{(3 p-5)}{2} \omega \mu
\ge -\frac{5p-7}{4} c_{\infty}(\mu)>0,
\]
from which we conclude. 
\end{proof}

\begin{proof}[Proof of Theorem \ref{thm:1.1}]
It is a direct consequence of Lemmas \ref{lem:3.1}, \ref{lem:3.4}, \ref{lem:3-ex},
\ref{lem:3.6} and \ref{lem:3.8}. 
\end{proof}

\section{Global well-posedness of the Cauchy problem} 

In this section, we consider the solvability of the Cauchy problem:
\begin{equation} \label{eq:4.1}
\begin{cases}
i \psi_t + \Delta \psi - e\phi \psi + |\psi|^{p-1} \psi =0 
\quad \hbox{in} \ \R_+ \times \R^3, \\
-\Delta \phi = \frac{e}{2}\big( |\psi|^2-\rho(x) \big) 
\quad \hbox{in} \ \R_+ \times \R^3, \\
\psi(0,x)=\psi_0,
\end{cases}
\end{equation}
where $e>0$, $1<p<5$ and $\psi_0 \in H^1(\R^3,\C)$.
For the doping profile $\rho$, we only assume that $\rho \in L^{\frac{6}{5}}(\R^3)$.
Then we have the following result on the local well-posedness.

\begin{proposition} \label{prop:4.1}
There exists $T=T(\| \psi_0\|_{H^1(\R^3)})>0$ such that 
\ef{eq:4.1} has a unique solution $\psi \in X$, where
\[
X= \left\{ \psi \in C\left( [0,T],H^1(\R^3,\C) \right) 
\cap L^{\infty}\left( (0,T), H^1(\R^3,\C) \right) \right\}.
\]
Furthermore, $\psi$ satisfies the energy conservation law 
and the charge conservation law:
\[
\mathcal{E}\big( \psi(t) \big) = \mathcal{E}(\psi_0)
\quad \hbox{and} \quad 
\| \psi(t) \|_2 = \| \psi_0 \|_2 \quad \hbox{for all} \ t \in (0,T).
\]
\end{proposition}

Although it seems that Proposition \ref{prop:4.1} can be obtained 
in the framework of \cite[Proposition 3.2.9, Theorem 4.3.1 and Corollary 4.3.3]{Ca},
we give the proof for the sake of completeness and reader's convenience.
For this purpose, we first recall the following inequalities in $\R^3$.

\begin{lemma}[Strichartz's estimates]  \label{lem:4.2}
Let $e^{it \Delta}$ be the linear propagator generated by the free
Schr\"odinger equation $i \psi_t + \Delta \psi =0$.
Assume that pairs $(q_0,r_0)$, $(q_1,r_1)$, $(q_2,r_2)$ are admissible, namely,
they fulfill the relation:
\[
\frac{2}{q_i} = 3 \left( \frac{1}{2} - \frac{1}{r_i} \right), \quad 
2 \le r_i \le 6 \quad (i=0,1,2).
\]
\begin{enumerate}
\item[\rm(i)] There exists $C>0$ such that 
\[
\| e^{it \Delta} f \|_{L^{q_0}( \R, L^{r_0}(\R^3) )} 
\le C \| f \|_{L^2(\R^3)} \quad \hbox{for all} \ f \in L^2(\R^3).
\]
\item[\rm(ii)] Let $I \subset \R$ be an interval, $J=\bar{I}$ and $t_0 \in J$.
Then there exists $C>0$ independent of $I$ such that
\[
\left\| \int_{t_0}^t e^{i(t-s) \Delta} f(s) \,ds \right\|_{L^{q_1} ( I, L^{r_1}(\R^3) )}
\le C \| f \|_{L^{q_2'}( I, L^{r_2'}(\R^3))}
\quad \hbox{for all} \ f \in {L^{q_2'}( I, L^{r_2'}(\R^3))},
\]
where $q_2'$ and $r_2'$ are H\"older conjugate exponents 
of $q_2$ and $r_2$ respectively.
\end{enumerate}

\end{lemma}

\begin{lemma}[Hardy-Littlewood-Sobolev inequality] \label{lem:4.3}
Let $0<\alpha < 3$ and $1<q<r<\infty$ satisfy
$\frac{1}{q} - \frac{1}{r} = 1- \frac{\alpha}{3}$,
and define $I_{\alpha}f$ by
\[
I_{\alpha} f(x) := \intR |x-y|^{-\alpha} f(y) \,dy.
\]
Then there exists $C>0$ such that 
\[
\| I_{\alpha} f \|_r \le C \| f \|_q 
\quad \hbox{for all} \ f \in L^q(\R^3).
\]
\end{lemma}

For simplicity, we write $L^{q}(I, L^{r}(\R^3))$ as $L_t^qL_x^r(I \times \R^3)$
and use a notation $A \lesssim B$ if there is a positive constant  
independent of $A$ and $B$ such that $A \le CB$.

\begin{proof}[Proof of Proposition \ref{prop:4.1}]
First as we have seen in Subsection 2.1, 
the Poisson equation \ef{eq:2.1} has a unique solution 
\[
\psi= eS(\psi) \in D_x^{1,2}(\R^3) \quad \hbox{for any} \ \psi \in H_x^1(\R^3).
\]

Let us consider the Duhamel formula associated with \ef{eq:4.1}
and the solution map $\mathcal{H}(\psi)$ which is defined by
\begin{equation} \label{eq:4.2}
\mathcal{H}(\psi) := e^{it \Delta} \psi_0 
+i \int_0^t e^{i(t-s) \Delta} |\psi|^{p-1} \psi \,ds
-e^2 i \int_0^t e^{i(t-s) \Delta} S(\psi) \psi \,ds. 
\end{equation}
For $T>0$ and $M= 3 \| \psi_0 \|_{H^1(\R^3)}$, 
we also define a complete metric space
\[
\begin{aligned}
X_T &:= \Big\{ \psi \in L_t^{\infty}H_x^1 \left( (-T,T)\times \R^3 \right)
\cap L_t^{\frac{4(p+1)}{3(p-1)}} W_x^{1,p+1} \left( (-T,T) \times \R^3 \right) \ ; \ \\
&\qquad 
\| \psi \|_{L_t^{\infty} H_x^1} \le M, \ 
\| \psi \|_{L_t^{\frac{4(p+1)}{3(p-1)}} W_x^{1,p+1}} \le M \Big\}
\end{aligned}
\]
equipped with the distance
\[
d(\psi_1, \psi_2) 
= \| \psi_1 -\psi_2 \|_{L_t^{\frac{4(p+1)}{3(p-1)}} L_x^{p+1}}
+ \| \psi_1-\psi_2 \|_{L_t^{\infty}L_x^2}. 
\]
Note that $\left( \frac{4(p+1)}{3(p-1)}, p+1 \right)$ and $(\infty, 2)$ 
are both admissible.
For simplicity, we write $q= \frac{4(p+1)}{3(p-1)}$.
It suffices to show that $\mathcal{H}$ is a contraction mapping on $X_T$
provided that $T$ is sufficiently small.

First we establish that $\mathcal{H}$ maps $X_T$ into itself.
To this aim, we apply Lemma \ref{lem:4.2} to find that 
\begin{align*}
\| \mathcal{H}(\psi) \|_{L_t^{\infty}H_x^1([0,T]\times \R^3)}
& \lesssim \| \psi_0 \|_{H^1(\R^3)}
+ \| S(\psi) \psi \|_{L_t^{\frac{4}{3}} W_x^{1, \frac{3}{2}} ( [0,T] \times \R^3)}
+ \| |\psi|^{p-1} \psi \|_{L_t^{q'}W_x^{1,\frac{p+1}{p}}([0,T] \times \R^3)}, \\
\| \mathcal{H}(\psi) \|_{L_t^{q}W_x^{1,p+1}([0,T]\times \R^3)}
& \lesssim \| \psi_0 \|_{H^1(\R^3)}
+ \| S(\psi) \psi \|_{L_t^{\frac{4}{3}} W_x^{1, \frac{3}{2}} ( [0,T] \times \R^3)}
+ \| |\psi|^{p-1} \psi \|_{L_t^{q'}W_x^{1,\frac{p+1}{p}}([0,T] \times \R^3)}.
\end{align*}
Here we mention that $\left( \frac{4}{3} \right)'= 4$, $\left( \frac{3}{2} \right)' = 3$
and the pair $(4,3)$ is admissible.
By the H\"older inequality and the Sobolev inequality, it follows that
\begin{align} \label{eq:4.3}
\| |\psi|^{p-1} \psi \|_{L_t^{q'} L_x^{\frac{p+1}{p}}}
&= \left( \int_0^T \| \psi \|_{L_x^{p+1}}^{q'} \| \psi \|_{L_x^{p+1}}^{(p-1)q'} \,dt
\right)^{\frac{1}{q'}} 
\le \| \psi \|_{L_t^qL_x^{p+1}} 
\left( \int_0^T \| \psi \|_{L_x^{p+1}}^{\frac{2(p-1)(p+1)}{5-p}} \,dt 
\right)^{\frac{5-p}{2(p+1)}} \notag \\
&\lesssim T^{\frac{5-p}{2(p+1)}} \| \psi \|_{L_t^qL_x^{p+1}} 
\| \psi \|_{L_t^{\infty} H_x^1}^{p-1} \le M^p T^{\frac{5-p}{2(p+1)}},
\end{align}
where we used the fact $q'\left( 1- \frac{q'}{q} \right)^{-1} = \frac{2(p+1)}{5-p}$.
Similarly one has
\begin{align} \label{eq:4.4}
\| \nabla (|\psi|^{p-1} \psi) \|_{L_t^{q'} L_x^{\frac{p+1}{p}}}
&\lesssim \| |\psi|^{p-1} \nabla \psi \|_{L_t^{q'} L_x^{\frac{p+1}{p}}} \notag \\
&\lesssim T^{\frac{5-p}{2(p+1)}} \| \nabla \psi \|_{L_t^qL_x^{p+1}} 
\| \psi \|_{L_t^{\infty} H_x^1}^{p-1} 
\le M^p T^{\frac{5-p}{2(p+1)}}.
\end{align}
Next by Lemma \ref{lem:2.1} and the fact $S(\psi)=S_1(\psi)+S_2$, we have
\begin{align} \label{eq:4.5}
\| S(\psi) \psi \|_{L_t^{\frac{4}{3}} L_x^{\frac{3}{2}}} 
&\le \left( \int_0^T \| S(\psi) \|_{L_x^6}^{\frac{4}{3}} 
\| \psi \|_{L_x^2}^{\frac{4}{3}} \,dt \right)^{\frac{3}{4}} 
\lesssim T^{\frac{3}{4}} \| S(\psi) \|_{L_t^{\infty} L_x^6} 
\| \psi \|_{L_t^{\infty} L_x^2} \notag \\
&\lesssim T^{\frac{3}{4}} \left( \| \psi\|_{L_t^{\infty}H_x^1}^2 
+ \| \rho \|_{\frac{6}{5}} \right) \| \psi \|_{L_t^{\infty} L_x^2} 
\le M(M^2+1) T^{\frac{3}{4}},
\end{align}
\begin{align} \label{eq:4.6}
\| \nabla ( S(\psi) \psi) \|_{L_t^{\frac{4}{3}} L_x^{\frac{3}{2}}}
&\le \| \psi \nabla S(\psi) \|_{L_t^{\frac{4}{3}} L_x^{\frac{3}{2}}}
+ \| S(\psi) \nabla \psi \|_{L_t^{\frac{4}{3}} L_x^{\frac{3}{2}}} \notag \\
&\lesssim T^{\frac{3}{4}} \| \nabla S(\psi) \|_{L_t^{\infty} L_x^2}
\| \psi \|_{L_t^{\infty} L_x^6} 
+T^{\frac{3}{4}} \| S(\psi) \|_{L_t^{\infty} L_x^6} \| \nabla \psi \|_{L_t^{\infty} L_x^2} 
\notag \\
&\lesssim T^{\frac{3}{4}} ( \| \psi \|_{L_t^{\infty}H_x^1}^2 +\| \rho \|_{\frac{6}{5}})
\| \psi \|_{L_t^{\infty} H_x^1}
\le M(M^2+1) T^{\frac{3}{4}}.
\end{align}
Thus from \ef{eq:4.3}-\ef{eq:4.6}, one finds that
\begin{align*}
\| \mathcal{H}(\psi) \|_{L_t^{\infty}H_x^1([0,T]\times \R^3)}
\lesssim \frac{M}{3} + M^p T^{\frac{5-p}{2(p+1)}} 
+M(M^2+1) T^{\frac{3}{4}}, \\
\| \mathcal{H}(\psi) \|_{L_t^{q}W_x^{1,p+1}([0,T]\times \R^3)}
\lesssim \frac{M}{3} + M^p T^{\frac{5-p}{2(p+1)}} 
+M(M^2+1) T^{\frac{3}{4}}.
\end{align*}
Choosing $T$ sufficiently small, it follows that
\[
\| \mathcal{H}(\psi) \|_{L_t^{\infty}H_x^1([0,T]\times \R^3)} \le M, \quad 
\| \mathcal{H}(\psi) \|_{L_t^q W_x^{1,p+1}([0,T]\times \R^3)} \le M
\]
and hence $\mathcal{H}$ maps $X_T$ into itself.

Finally we prove that $\mathcal{H}$ is a contraction mapping on $X_T$.
Applying Lemma \ref{lem:4.2} again, we first observe that
\[
d \big( \mathcal{H}(\psi_1), \mathcal{H}(\psi_2) \big)
\lesssim \| |\psi_1|^{p-1}\psi_1 - | \psi_2|^{p-1} \psi_2 
\|_{L_t^{q'} L_x^{\frac{p+1}{p}}( [0,T]\times \R^3)}
+\| S(\psi_1)\psi_1 - S(\psi_2)\psi_2 
\|_{L_t^{\frac{4}{3}} L_x^{\frac{3}{2}} ([0,T]\times \R^3)}.
\] 
By the H\"older inequality, it follows that
\begin{align} \label{eq:4.7}
\| |\psi_1|^{p-1}\psi_1 - | \psi_2|^{p-1} \psi_2 
\|_{L_t^{q'} L_x^{\frac{p+1}{p}}}
&\lesssim \left( \int_0^T \| \psi_1-\psi_2 \|_{L_x^{p+1}}^{q'}
\left( \| \psi_1 \|_{L_x^{p+1}}^{p-1} + \| \psi_2 \|_{L_x^{p+1}}^{p-1} \right)^{q'}
\,dt \right)^{\frac{1}{q'}} \notag \\
&\lesssim T^{\frac{5-p}{2(p+1)}} \| \psi_1-\psi_2 \|_{L_t^q L_x^{p+1}}
\left( \| \psi \|_{L_t^{\infty} H_x^1}^{p-1} + \| \psi_2 \|_{L_t^{\infty} H^1_x}^{p-1}
\right) \notag \\
&\le M^{p-1} T^{\frac{5-p}{2(p+1)}} d(\psi_1, \psi_2). 
\end{align}
Moreover one has
\begin{align*}
&\| S(\psi_1)\psi_1- S(\psi_2)\psi_2 \|_{L_t^{\frac{4}{3}} L_x^{\frac{3}{2}}}
\le \| S(\psi_1) (\psi_1-\psi_2) \|_{L_t^{\frac{4}{3}} L_x^{\frac{3}{2}}}
+ \| (S(\psi_1)-S(\psi_2)) \psi_2 \|_{L_t^{\frac{4}{3}} L_x^{\frac{3}{2}}} \\
&\lesssim T^{\frac{3}{4}} \| S(\psi_1) \|_{L_t^{\infty} L_x^6}
\| \psi_1- \psi_2 \|_{L_t^{\infty} L_x^2}
+T^{\frac{3}{4}} \| S(\psi_1)-S(\psi_2) \|_{L_t^{\infty} L_x^6}
\| \psi_2 \|_{L_t^{\infty} L_x^2}.
\end{align*}
Here we recall that
\begin{align*}
S(\psi_1)-S(\psi_2)
&= S_1(\psi_1)+S_2 - S_1(\psi_2)-S_2 
= \frac{1}{8\pi |x|} * ( |\psi_1|^2-|\psi_2|^2) \\
&= \frac{1}{8\pi} \intR |x-y|^{-1} ( |\psi_1(y)|^2-|\psi_2(y)|^2) \,dy.
\end{align*}
Applying Lemma \ref{lem:4.3} with $\alpha=1$, $q= \frac{6}{5}$ and $r=6$,
one gets
\[
\| S(\psi_1)-S(\psi_2) \|_{L_x^6}
\lesssim \| |\psi_1|^2-|\psi_2|^2 \|_{L_x^{\frac{6}{5}}} 
\lesssim \| \psi_1- \psi_2 \|_{L_x^2} ( \| \psi_1 \|_{L_x^3} + \| \psi_2 \|_{L_x^3}),
\]
from which we find that
\begin{equation} \label{eq:4.8}
\| S(\psi_1) -S(\psi_2) \|_{L_t^{\frac{4}{3}} L_x^{\frac{3}{2}}}
\lesssim M^2 T^{\frac{3}{4}} \| \psi_1 - \psi_2 \|_{L_t^{\infty} L_x^2}
\le M^2T^{\frac{3}{4}} d(\psi_1- \psi_2).
\end{equation}
From \ef{eq:4.7} and \ef{eq:4.8}, it follows that
\[
d \big( \mathcal{H}(\psi_1), \mathcal{H}(\psi_2) \big)
\le \frac{1}{2} d(\psi_1-\psi_2),
\]
provided that $T$ is sufficiently small.

Since $\mathcal{H}$ is a contraction mapping,
there exists a unique fixed point $\psi$, that is, $\psi$ satisfies \ef{eq:4.2}. 
By the Strichartz estimate, we can also find that 
$\mathcal{H}(\psi) \in C([0,T],H^1(\R^3,\C))$
(see \cite[Theorem 2.3.3]{Ca})
and hence $\psi \in C([0,T],H^1(\R^3,\C))$.
Thus $\psi$ is a unique solution on \ef{eq:4.1} in $X$.
Once we could obtain the local well-posedness in $H^1$,
a standard argument shows that 
the energy conservation law and the charge conservation law hold.
\end{proof}

Finally in this section, we prove the following global well-posedness result
in the $L^2$-subcritical case,
which is a direct consequence of the conservation laws.

\begin{proposition} \label{prop:4.4}
Suppose that $1<p<\frac{7}{3}$. 
Then the unique solution $\psi$ obtained in Proposition \ref{prop:4.1}
exists globally in $t>0$.
\end{proposition}

\begin{proof}
It is sufficient to show that there exists $C>0$ independent on $t$ such that
$\| \nabla \psi(t) \|_{L_x^2} \le C$ for all $t$ in the existence interval.

Now by the definition of the energy $\mathcal{E}$,
the Gagliardo-Nirenberg inequality and Lemma \ref{lem:2.1}, one has
\begin{align*}
\| \nabla \psi \|_{L_x^2}^2
&= 2 \mathcal{E}(\psi) + \frac{1}{p+1} \| \psi \|_{L_x^{p+1}}^{p+1} 
-e^2 A_1(\psi) -2e^2 A_2(\psi) -e^2 A_0  \\
&\lesssim 2\mathcal{E}(\psi) 
+ \| \psi \|_{L_x^2}^{\frac{5-p}{2}} 
\| \nabla\psi  \|_{L_x^2}^{\frac{3(p-1)}{2}}
+ \| \psi \|_{L_x^2}^3 \| \nabla \psi \|_{L_x^2} 
+ \| \rho \|_{\frac{6}{5}} \| \psi \|_{L_x^2}^{\frac{3}{2}} 
\| \nabla \psi \|_{L_x^2}^{\frac{1}{2}}
+ \| \rho\|_{\frac{6}{5}}^2. 
\end{align*}
Moreover using the Young inequality and the two conservations laws,
for any $\ep' \in (0,1)$, we deduce that
\[
\| \nabla \psi \|_{L_x^2}^2
\lesssim 2 \mathcal{E}(\psi_0) + \ep' \| \nabla \psi \|_{L_x^2}^2
+ \| \psi_0 \|_2^{\frac{2(5-p)}{7-3p}}
+ \| \psi_0 \|_{2}^6 + \| \rho \|_{\frac{6}{5}} \| \psi_0 \|_{2}
+ \| \rho \|_{\frac{6}{5}}^2,
\]
from which we conclude.
\end{proof}

\section{Stability of standing waves}

In this section, we prove the orbital stability of standing waves associated 
with minimizers for $C(\mu)$, which is a direct consequence of Lemma \ref{lem:3.6}.

\begin{proof}[Proof of Theorem \ref{thm:1.2}]
The proof follows the argument of \cite{CL}.
First we observe, 
since $\phi(t,\cdot)= \frac{e}{2} (-\Delta)^{-1} (|\psi(t,\cdot)|^2-\rho)$, 
that if
\begin{align*}
& \sup_{t>0} \Bigl\{ \inf_{y \in \R^3,} \inf_{u \in \mathcal{M}(\mu)} 
\big\| \psi(t,\cdot) - u (\cdot +y) \big\|_{H^1} \Big\} 
< \varepsilon,
\end{align*}
one also has
\[
\sup_{t>0} \inf_{y \in \R^3,} \inf_{u \in \mathcal{M}(\mu) }
\left\| \phi(t,\cdot) - \frac{e}{2} (-\Delta)^{-1}
 (|u(\cdot +y)|^2- \rho) \right\|_{D^{1,2}}
< C \varepsilon
\]
for some $C>0$ independent of $\varepsilon$. 
Thus it is enough to prove that for every $\varepsilon>0$, 
there exists $\delta(\varepsilon)>0$ such that
for any initial data $\psi_{(0)}$ satisfying 
\[
\inf_{u \in \mathcal{M}(\mu)} \| \psi_{(0)} - u \|_{H^1} < \delta,
\]
the corresponding solution $\psi$ verifies
\begin{align*}
\sup_{t>0} 
\inf_{y \in \R^3,} \inf_{u \in \mathcal{M}(\mu) }
\| \psi(t,\cdot)- u(\cdot +y) \|_{H^1} < \varepsilon.
\end{align*}
For that purpose, we assume by contradiction that there exist $\varepsilon_0>0$,
\[
\big( \psi_{(0)j} \big)_{j \in \N}
\subset H^1(\R^3,\C) 
\]
and $\{ t_j \} \subset \R$ such that 
\begin{equation} \label{eq:5.1}
\inf_{u \in \mathcal{M}(\mu)} \| \psi_{(0)j}- u \|_{H^1}
\to 0 \ \hbox{as} \ j \to \infty,
\end{equation}
but the corresponding solution $(\psi_j)$ satisfies
\begin{align} \label{eq:5.2}
\inf_{y \in \R^3,} \inf_{u \in \mathcal{M}(\mu) }
\| \psi(t_j,\cdot)- u(\cdot +y) \|_{H^1} \ge \varepsilon_0.
\end{align}
For simplicity, we write $u_j=\psi_j(t_j,\cdot)$. 
Then by the charge conservation law and from \ef{eq:5.1}, it follows that
\begin{equation} \label{eq:5.3}
\| u_j \|_{2}^2 = \| \psi_{(0)j} \|_{2}^2 \to \mu.
\end{equation}
By the energy conservation law, we also have
\begin{align*}
\mathcal{E}\big( u_j \big)
&= \mathcal{E}\big( \psi_{(0)j} \big) \\
&= \frac{1}{2} \IT | \nabla \psi_{(0)j} |^2 \,dx 
+e^2 A_1(\psi_{(0)j}) +2e^2 A_2(\psi_{(0)j}) +e^2 A_0
-\frac{1}{p+1} \IT | \psi_{(0)j}|^{p+1} \,dx.
\end{align*}
From \ef{eq:5.1}, one gets
\begin{equation} \label{eq:5.4}
\mathcal{E}\big( u_j \big)
\to \mathcal{C} \big( \mu \big).
\end{equation}
From \ef{eq:5.3}, \ef{eq:5.4} and
by Lemma \ref{lem:3.6}, 
there exist $u_{\mu} \in \mathcal{M}(\mu)$ and
$\{ y_j \} \subset \R^3$ such that
$u_j(\cdot) - u_{\mu}(\cdot+y_j) \to 0$ in $H^1(\R^3)$,
in contradiction with \ef{eq:5.2}. This ends the proof of Theorem \ref{thm:1.2}.
\end{proof}

\section{The case $\rho$ is a characteristic function}

In this section, we consider the case where the doping profile $\rho$
is a characteristic function, 
which appears frequently in physical literatures \cite{Je, MRS, Sel}.
More precisely, let $\{ \Omega_i \}_{i=1}^m \subset \R^3$ be
disjoint bounded open sets with smooth boundary.
For $\alpha_i >0$ $(i=1,\cdots, m)$, 
we assume that the doping profile $\rho$ has the form:
\begin{equation} \label{eq:6.1}
\rho(x)= \sum_{i=1}^m \alpha_i \chi_{\Omega_i}(x), \quad
\chi_{\Omega_i}(x)=
\begin{cases} 
1 & (x \in \Omega_i), \\
0 & (x \notin \Omega_i).
\end{cases}
\end{equation}
In this case, $\rho$ cannot be weakly differentiable so that the assumption
$\| x \cdot \nabla \rho \|_{\frac{6}{5}} \le \rho_0$ does not make sense.
Even so, we are able to obtain the existence of stable standing waves
by imposing some smallness condition related with $\Omega_i$.

To state our main result for this case, 
let us put $\dis L:= \sup_{x \in \partial \Omega} |x| < \infty$.
A key is the following \textit{sharp boundary trace inequality} 
due to \cite[Theorem 6.1]{A},
which we present here according to the form used in this paper.

\begin{proposition} \label{prop:6.1}
Let $\Omega \subset \R^3$ be a bounded domain with smooth boundary
and $\gamma: H^1(\Omega) \to L^2(\partial \Omega)$ be 
the trace operator.
Then it holds that
\[
\int_{\partial \Omega} | \gamma(u) |^2 \,dS
\le \kappa_1(\Omega) \int_{\Omega} |u|^2 \,dx
+ \kappa_2(\Omega) 
\left( \int_{\Omega} |u|^2 \,dx \right)^{\frac{1}{2}}
\left( \int_{\Omega} | \nabla u|^2 \,dx \right)^{\frac{1}{2}}
\quad \hbox{for any} \ u \in H^1(\Omega),
\]
where $\kappa_1(\Omega)= \frac{|\partial \Omega|}{|\Omega|}$,
$\kappa_2(\Omega) = \big\| | \nabla w| \big\|_{L^{\infty}(\partial\Omega)}$
and $w$ is a unique solution of the torsion problem:
\[
\Delta w = \kappa_1(\Omega) \ \hbox{in} \ \Omega, 
\quad \frac{\partial w}{\partial n} = 1 \ \hbox{on} \ \partial \Omega.
\]
\end{proposition}

In relation to the size of $\rho$, we define
\[
D(\Omega) := L | \Omega |^{\frac{1}{6}} | \partial \Omega |^{\frac{1}{2}}
\left( \kappa_1(\Omega) | \Omega |^{\frac{1}{3}} + \kappa_2 (\Omega) 
\right)^{\frac{1}{2}}.
\]

\begin{remark} \label{rem:6.2}
It is known that $\kappa_2(\Omega) \ge 1$; see \cite{A}.
Then by the isoperimetric inequality in $\R^3$:
\[
| \partial \Omega | \ge 3 | \Omega |^{\frac{2}{3}} | B_1 |^{\frac{1}{3}},
\]
and the fact $| \Omega | \le |B_L(0)| = L^3 |B_1|$, we find that
\begin{equation} \label{eq:6.2}
D(\Omega) \ge
\left( \frac{|\Omega|}{|B_1|} \right)^{\frac{1}{3}} |\Omega|^{\frac{1}{6}}
\cdot \sqrt{3} |\Omega|^{\frac{1}{3}} |B_1|^{\frac{1}{6}}
\left( 3 |B_1|^{\frac{1}{3}} +1 \right)^{\frac{1}{2}} 
=C |\Omega|^{\frac{5}{6}} = C \| \chi_{\Omega} \|_{L^{\frac{6}{5}}(\R^3)},
\end{equation}
where $C$ is a positive constant independent of $\Omega$. 
\end{remark}

Under these preparations, we have the following result.

\begin{theorem} \label{thm:6.3} 
Suppose that $2<p<\frac{7}{3}$, assume \ef{ASS} and 
let $\mu > 2\cdot 2^{\frac{1}{2p-4}} \mu^*$ be given. 
Under the assumption \ef{eq:6.1}, 
there exists $\rho_0= \rho_0(e, \mu)>0$ such that 
if $\dis \sum_{i=1}^m  \alpha_i D(\Omega_i) \le \rho_0$,
the minimization problem \ef{eq:1.5} admits a minimizer $u_{\mu}$.

Moreover the associated Lagrange multiplier $\omega=\omega(\mu)$ is positive.

\end{theorem}
Similarly to Theorem \ref{thm:1.2}, 
the orbital stability of $e^{i\omega t} u_{\mu}(x)$ also holds true.

\smallskip
We mention that the first time $x \cdot \nabla \rho(x)$ appeared
was the definition of $A_3(u)$ in Subsection 2.3.
Under the assumption \ef{eq:6.1}, we replace $A_3(u)$ by
\begin{equation} \label{eq:6.3}
A_3(u) := - \frac{1}{2} \sum_{i=1}^m \alpha_i 
\int_{\partial \Omega_i} S_1(u) x \cdot n_i \,dS_i,
\end{equation}
where $n_i$ is the unit outward normal on $\partial \Omega_i$.
Indeed we have the following.

\begin{lemma} \label{lem:6.4}
Let $\Omega \subset \R^3$ be a bounded domain with smooth boundary
and put
\[
\Omega(\lambda) 
:= \int_{\mathbb{R}^3} S_1(u)(x) \chi_{\Omega}\left( \frac{x}{\lambda}\right) \,d x
\quad \text {for} \ \lambda>0.
\]
Then it holds that
\[
\Omega^{\prime}(\lambda)
= \frac{1}{\lambda} \int_{\partial (\lambda \Omega)} S_1(u)(x) x \cdot n \,dS. 
\]
\end{lemma}

\begin{proof}
The proof is based on the domain deformation as in \cite{XLW}.
In fact, one has 
\begin{align*}
\frac{d}{d\lambda} \Omega(\lambda) 
&= \frac{d}{d\lambda} \int_{\lambda \Omega} S_1(u)(x) \,dx \\
&= \lim_{h \to 0} \frac{1}{h}
\left( \int_{(\lambda+h)\Omega} S_1(u)(x) \,dx 
- \int_{\lambda \Omega} S_1(u)(x) \,dx \right) \\
&= \lim_{h \to 0} \frac{1}{h} 
\int_{(\lambda+h)\Omega \setminus (\lambda \Omega)} S_1(u)(x) \,dx \\
&= \lim_{h \to 0} \frac{1}{h} \int_{\lambda}^{\lambda+h} 
\left( \frac{1}{t} \int_{\partial (t \Omega)} S_1(u)(x) 
x \cdot (t n) \,dS \right) \,dt \\
&= \frac{1}{\lambda} \int_{\partial (\lambda \Omega)} S_1(u)(x) x \cdot n \,dS,
\end{align*}
from which we conclude.
\end{proof}

By Lemma \ref{lem:6.4}, the Pohozev identity can be reformulated as follows.

\begin{lemma} \label{lem:6.5}
Under the assumption \ef{eq:6.1}, any nontrivial solution $u$ of \ef{eq:1.1}
satisfies the following identity.
\[
0 = \frac{1}{2} \| \nabla u \|_2^2 + \frac{3\omega}{2} \| u \|_2^2
- \frac{3}{p+1} \| u \|_{p+1}^{p+1} 
+ 5e^2 A_1(u) + 10e^2 A_2(u) 
+ \frac{e^2}{2} \sum_{i=1}^m \alpha_i 
\int_{\partial \Omega_i} S_1(u) x \cdot n_i \,dS_i, 
\]
\[
\begin{aligned}
(5 p-7) E(u) &= 
2(p-2) \|\nabla u\|_2^2 -\frac{(3 p-5) \omega}{2} \| u\|_2^2 \\
&\quad -e^2 \sum_{i=1}^m \alpha_i \left( 
2 \int_{\Omega_i} S_1(u) \,dx 
-\frac{3-p}{2} \int_{\partial \Omega_i} S_1(u) x \cdot n_i \,dS_i \right).
\end{aligned}
\]

\end{lemma}

\begin{proof}
As we have seen in Subsection 2.5, 
the Pohozaev identity can be obtained by considering 
$\frac{d}{d\lambda} I(u_{\lambda}) |_{\lambda =1}$ with 
$u_{\lambda}(x)= u \left( \frac{x}{\lambda} \right)$.
Applying Lemma \ref{lem:6.4} with $\lambda=\lambda^{-1}$, we then obtain
\begin{align*} 
0&=\frac{1}{2} \| \nabla u \|_2^2 
+\frac{3 \omega}{2}\|u\|_2^2
-\frac{3}{p+1} \| u \|_{p+1}^{p+1} 
+5 e^2 A(u)
+\frac{e^2}{2} \sum_{i=1}^m \alpha_i \int_{\partial \Omega_i}
S(u) x \cdot n_i \,dS_i \\
&\quad -5 e^2 A_0 - \frac{e^2}{2} \sum_{i=1}^m \alpha_i
\int_{\partial \Omega_i} S_2 x \cdot n_i \,dS_i.
\end{align*}
Now recalling that
\begin{align*}
A_0 &= - \frac{1}{4} \intR S_2 \rho(x) \,dx
= - \frac{1}{4} \sum_{i=1}^m \alpha_i \int_{\Omega_i} S_2 \,dx, \\
S_2(x) &= -\frac{1}{8\pi} \intR \frac{\rho(y)}{|x-y|} \,dy
= - \frac{1}{8\pi} \sum_{i=1}^m \alpha_i 
\int_{\Omega_i} \frac{1}{|x-y|} \,dy,
\end{align*}
one finds that
\begin{align*}
\int_{\Omega_i} x \cdot \nabla S_2 \,dx
&= \frac{1}{8\pi} \sum_{i=1}^m \alpha_i \int_{\Omega_i} \int_{\Omega_i}
\frac{x \cdot (x-y)}{|x-y|^3} \,dy \,dx 
= \frac{1}{8\pi} \sum_{i=1}^m \alpha_i \int_{\Omega_i} \int_{\Omega_i}
\left( \frac{1}{|x-y|} + \frac{y \cdot (x-y)}{|x-y|^3} \right) \,dy\,dx  \\
&= - \int_{\Omega_i} S_2 \,dx 
- \int_{\Omega_i} x \cdot \nabla S_2 \,dx,
\end{align*}
and hence
\[
\int_{\Omega_i} x \cdot \nabla S_2 \,dx 
= - \frac{1}{2} \int_{\Omega_i} S_2 \,dx.
\]
Here we used the Fubini theorem.
Thus by the divergence theorem, we get
\begin{align*}
-5e^2 A_0 - \frac{e^2}{2} \sum_{i=1}^m \alpha_i
\int_{\partial \Omega_i} S_2 x \cdot n_i \,dS_i 
&= \frac{5e^2}{4} \sum_{i=1}^m \alpha_i 
\int_{\Omega_i} S_2 \,dx 
- \frac{e^2}{2} \sum_{i=1}^m \alpha_i
\int_{\Omega_i} {\rm div}(x S_2) \,dx \\
&= - \frac{e^2}{4} \sum_{i=1}^m \alpha_i 
\int_{\Omega_i} S_2 \,dx
- \frac{e^2}{2} \sum_{i=1}^m \alpha_i
\int_{\Omega_i} x \cdot \nabla S_2 \,dx =0.
\end{align*}
This completes the proof of the Pohozaev identity.
Then similarly to Lemma \ref{lem:2.3}, 
we can show the second identity.
\end{proof}

Next we establish estimates for $A_2$ and $A_3$.

\begin{lemma} \label{lem:6.6}
For any $u \in H^1(\R^3,\C)$, $A_2$ and $A_3$ satisfy the estimates:
\begin{align*}
|A_2(u)| &\le C \sum_{i=1}^m \alpha_i | \Omega_i |^{\frac{5}{6}}
\| u\|_2^{\frac{3}{2}} \| \nabla u \|_2^{\frac{1}{2}}, \\
|A_3(u)| &\le C \sum_{i=1}^m \alpha_i D(\Omega_i) 
\| u\|_2^{\frac{3}{2}} \| \nabla u \|_2^{\frac{1}{2}}, 
\end{align*}
where $C>0$ is a constant independent of $\Omega_i$.
\end{lemma}

\begin{proof}
First we observe that 
\[
|A_2(u)| \le \frac{1}{4} \sum_{i=1}^m \alpha_i \int_{\Omega_i} |S_1(u)|\,dx,
\]
from which the estimate for $A_2$ can be obtained by the H\"older inequality
and Lemma \ref{lem:2.1}.
Next by Proposition \ref{prop:6.1}, 
the H\"older inequality and the Sobolev inequality, one has
\begin{align*}
|A_3(u)| & \le \frac{1}{2} \sum_{i=1}^m \alpha_i
\int_{\partial \Omega_i} |S_1(u)| |x| \,dS_i \\
&\le \frac{1}{2} \sum_{i=1}^m \alpha_i
\left( \int_{\partial \Omega_i} |S_1(u)|^2 \,dS_i \right)^{\frac{1}{2}}
\left( \int_{\partial \Omega_i} |x|^2 \,dS_i \right)^{\frac{1}{2}} \\
&\le \frac{1}{2} \sum_{i=1}^m \alpha_i L_i | \partial \Omega_i|^{\frac{1}{2}}
\left( \kappa_1(\Omega_i) \| S_1(u) \|_{L^2(\Omega_i)}^2
+\kappa_2(\Omega_i) \| S_1(u) \|_{L^2(\Omega_i)}
\| \nabla S_1(u) \|_{L^2(\Omega_i)} \right)^{\frac{1}{2}} \\
&\le \frac{1}{2} \sum_{i=1}^m \alpha_i L_i |\partial \Omega_i|^{\frac{1}{2}}
\left( \kappa_1(\Omega_i) |\Omega_i|^{\frac{2}{3}} \| S_1(u) \|_{L^6(\R^3)}^2
+\kappa_2(\Omega_i) | \Omega_i |^{\frac{1}{3}} \| S_1(u) \|_{L^6(\R^3)}
\| \nabla S_1(u) \|_{L^2(\R^3)} \right)^{\frac{1}{2}} \\
&\le C \sum_{i=1}^m \alpha_i L_i |\Omega_i|^{\frac{1}{6}}
|\partial \Omega_i|^{\frac{1}{2}} 
\left( \kappa_1(\Omega_i)|\Omega_i|^{\frac{1}{3}} 
+\kappa_2(\Omega_i) \right)^{\frac{1}{2}} 
\| \nabla S_1(u) \|_{L^2(\R^3)} \\
&\le C \sum_{i=1}^m \alpha_i D(\Omega_i) \| u \|_2^{\frac{3}{2}}
\| \nabla u \|_2^{\frac{1}{2}}.
\end{align*}
This completes the proof.
\end{proof}

Now we are ready to prove Theorem \ref{thm:6.3}.

\begin{proof}[Proof of Theorem \ref{thm:6.3}]
We establish the existence of a minimizer for $2<p<\frac{7}{3}$.
For this purpose, it suffices to modify the proof of Lemma \ref{lem:3.4} only,
because the other part of the existence proof 
does not rely on $x \cdot \nabla \rho(x)$.
Under the same notation as the proof of Lemma \ref{lem:3.4}
and applying Lemma \ref{lem:6.4}, we arrive at 
\begin{align*} 
f^{\prime}(\lambda)
&= \frac{p-1+(3p-7)b}{2} \lambda^{\frac{p-3+(3p-7)b}{2}} E(u_{\ep}) 
- \frac{p-1+(3p-7)b}{4} \lambda^{\frac{p-3+(3p-7)b}{2}} \| \nabla u_{\ep} \|_2^2 \\
&\quad - e^2 \left( \frac{p-1+(3p-7)b}{2} \lambda^{\frac{p-3+(3p-7)b}{2}}
- (1-b) \lambda^{-b} \right) A_1(u_{\ep}) \\
&\quad - e^2 \big( p-1 +(3p-7)b \big) \lambda^{\frac{p-3+(3p-7)b}{2}} A_2(u_{\ep}) 
+2 e^2 \frac{d}{d\lambda} 
\left( \lambda^{-1 -2b} A_2 \big( (u_{\ep})_\lambda \big) \right).
\end{align*}
Choosing $b=1$ for simplicity and using \ef{eq:2.3}, we infer that 
\[
\begin{aligned}
f'(\lambda) &\le 2(p-2) \lambda^{2p-5} E(u_{\ep}) 
-(p-2) \lambda^{2p-5} \| \nabla u_{\ep} \|_2^2
- 4 (p-2) e^2 \lambda^{2p-5} A_2(u_{\ep}) \\
&\quad +2 e^2 \frac{d}{d\lambda} 
\left( \lambda^{-3} A_2 \big( (u_{\ep})_\lambda \big) \right).
\end{aligned}
\]

Now recalling \ef{eq:6.1} and the definition of $\Omega(\lambda)$ in Lemma \ref{lem:6.4},
one finds that
\begin{align*} 
A_2 \big( (u_{\ep})_\lambda \big) 
&= -\frac{1}{4} \sum_{i=1}^m \alpha_i \int_{\mathbb{R}^3} 
S_1 \big( (u_{\ep})_{\lambda} \big)(x) \chi_{\Omega_i}(x) \,d x 
= -\frac{\lambda^{-1}}{4} \sum_{i=1}^m \alpha_i \int_{\mathbb{R}^3} 
S_1 (u_{\ep})(x) \chi_{\Omega_i}\left(\frac{x}{\lambda}\right) \,dx \notag \\
&= -\frac{\lambda^{-1}}{4} \sum_{i=1}^m \alpha_i \Omega_i (\lambda).
\end{align*}
Thus it follows that
\[
\begin{aligned}
2 e^2 \frac{d}{d \lambda} \left( \lambda^{-3} A_2 \big( (u_{\ep})_{\lambda} \big) \right)
&= -\frac{e^2}{2} \sum_{i=1}^m \alpha_i 
\frac{d}{d \lambda} \left( \lambda^{-4} \Omega_i(\lambda) \right) 
= 2  \lambda^{-5} e^2 \sum_{i=1}^m \alpha_i \Omega_i (\lambda)
-\frac{\lambda^{-4}}{2} e^2 \sum_{i=1}^m \alpha_i \Omega_i^{\prime}(\lambda).
\end{aligned}
\]
Moreover from \ef{eq:6.3} and by Lemma \ref{lem:6.4}, one has
\[
\begin{aligned}
\sum_{i=1}^m \alpha_i \Omega_i'(\lambda)
= \sum_{i=1}^m \int_{\partial \Omega_i}
\lambda^2 S_1(u_{\ep})(\lambda x) x \cdot n \,dS
= \sum_{i=1}^m \int_{\partial \Omega_i} S_1\big( (u_{\ep})_{\lambda} \big)(x) x \cdot n \,dS
&= -2 A_3 \big( (u_{\ep})_{\lambda} \big)
\end{aligned}
\]
and hence 
\[
2 e^2 \frac{d}{d \lambda} \left( \lambda^{-3} A_2 \big( (u_{\ep})_{\lambda} \big) \right)
= -8 \lambda^{-4} e^2 A_2 \big( (u_{\ep})_{\lambda} \big)
+ \lambda^{-4} e^2 A_3 \big( (u_{\ep})_{\lambda} \big).
\]
Finally by Lemma \ref{lem:2.1} and Lemma \ref{lem:6.6}, we deduce that
\[
\begin{aligned}
\left| A_2 \big( (u_{\ep})_{\lambda} \big) \right|
&\le C \sum_{i=1}^m \alpha_i | \Omega_i |^{\frac{5}{6}} 
\| (u_{\ep})_{\lambda} \|_2^{\frac{3}{2}}
\| \nabla (u_{\ep})_{\lambda} \|_2^{\frac{1}{2}} 
= C \lambda^{\frac{3}{2}} \mu^{\frac{3}{4}}
\sum_{i=1}^m \alpha_i |\Omega_i |^{\frac{5}{6}} 
\| \nabla u_{\ep} \|_2^{\frac{1}{2}}, \\
\left| A_3 \big( (u_{\ep})_{\lambda} \big) \right|
&\le C \lambda^{\frac{3}{2}} \mu^{\frac{3}{4}}
\sum_{i=1}^m \alpha_i D(\Omega_i) \| \nabla u_{\ep} \|_2^{\frac{1}{2}}.
\end{aligned}
\]
Using \ef{eq:6.2}, we obtain
\[
\begin{aligned}
f'(\lambda) &\le 2(p-2) \lambda^{2p-5}
\left\{ \frac{1}{4} c_{\infty} \left( \frac{\mu}{2} \right) 
- \frac{1}{8} \| \nabla u_{\ep} \|_2^{\frac{1}{2}} 
+ C_2 e^{\frac{8}{3}} \mu \left( \sum_{i=1}^m \alpha_i D(\Omega_i) \right)^{\frac{4}{3}}
\ \right\} \\
&\quad + C_1 \lambda^{-\frac{5}{2}} e^2 \mu^{\frac{3}{4}}
\sum_{i=1}^m \alpha_i D(\Omega_i) \| \nabla u_{\ep} \|_2^{\frac{1}{2}}.
\end{aligned}
\]
Thus similarly to Lemma \ref{lem:3.4},
there exists $\rho_0 = \rho(e, \mu)>0$ such that 
\[
\sum_{i=1}^m \alpha_i D(\Omega_i) \le \rho_0
\quad \Rightarrow \quad f'(\lambda)<0 \ \text{for all} \ \lambda >1
\ \text{and} \ s \in \left[ \frac{\mu}{2\cdot 2^{\frac{1}{2p-4}}}, \mu \right],
\]
from which we can conclude.

We can also show the other parts of Theorem \ref{thm:6.3}
by modifying the proof of Lemma \ref{lem:3.8} in a similar way. 
\end{proof}

\section{Concluding remark and an open problem}

In this paper, the nonlinear Schr\"odinger-Poisson system with a doping profile
has been investigated.
By establishing the existence of $L^2$-constraint minimizers when
\[ 
2<p<\frac{7}{3}, \ \mu > 2\cdot 2^{\frac{1}{2p-4}} \mu^* 
\quad \text{and} \quad 
\| \rho \|_{\frac{6}{5}} + \| x \cdot \nabla \rho \|_{\frac{6}{5}} \ll 1. 
\]
we are able to obtain stable standing waves.
The presence of a doping profile $\rho$ causes a difficulty of
proving the strict sub-additivity which is the key of 
the existence and the stability of standing waves.
This paper concludes by providing one open problem.

\medskip
\noindent
\textbf{Problem}: \ 
Non-existence of minimizers for large $\rho$ ?

\medskip
We have shown the existence of minimizers when $\rho$ is small,
but we don't know what happens if $\rho$ is large.
In 1D case, it was shown in \cite{DeLeo} that
no minimizer exists in the case $\mu < \| \rho \|_{L^1(\R)}$,
which was referred to the \textit{supercritical case}.
(Note that \cite{DeLeo} deals with the Schr\"odinger-Poisson system
with $\Delta \phi = \frac{1}{2}(|u|^2-\rho(x))$
so that the sign in the front of $A(u)$ in \ef{eq:1.4} is opposite.)
Hence a natural question is whether a similar result holds for the 3D problem.

To explain the idea in \cite{DeLeo}, 
let us consider the problem in $\R^N$ and denote by $G(x)$
the fundamental solution of $-\Delta$ on $\R^N$.
Under the assumption $\rho \ge 0$ and $\rho \in L^1(\R^N)$,
the nonlocal term $A(u)$ can be expressed as follows.
\begin{align*}
A(u) &= 
\frac{1}{8} \intR \intR G(x-y) \big( |u(x)|^2-\rho(x) \big)
\big( |u(y)|^2-\rho(y) \big) \,dx\,dy \\
&= \frac{1}{8} \intR \intR G(x-y) |u(x)|^2 (u(y)|^2 \,dx \,dy
-\frac{1}{4} \intR \intR G(x-y)|u(x)|^2 \rho(y) \,dx\,dy + A_0 \\
&= \frac{1}{8} \intR \intR \big( G(x-y)-G(x)-G(y) \big) 
|u(x)|^2 |u(y)|^2 \,dx \,dy \\
&\quad +\frac{1}{4} \big( \| u\|_{L^2(\R^N)}^2 - \| \rho \|_{L^1(\R^N)} \big)
\intR G(x) |u(x)|^2 \,dx \\
&\quad -\frac{1}{4} \intR \intR \big( G(x-y)-G(x) \big) |u(x)|^2 \rho(y) \,dx\,dy
+A_0.
\end{align*}
Here we have used Fubini's theorem and wrote 
$A_0= \frac{1}{8} \intR \intR G(x-y) \rho(x) \rho(y) \,dx \,dy$.
Thus one has $c(\mu)=-\infty$ if we could show that 
there is a family $\{ u_{\lambda} \} \subset H^1(\R^N)$ satisfying
\begin{align}
&\| u_{\lambda} \|_{L^2(\R^N)}^2= \mu, 
\quad \| \nabla u_{\lambda} \|_{L^2(\R^N)} \le C, \label{eq:7.1} \\
&\left| \intR \intR \big( G(x-y)-G(x)-G(y) \big) 
|u_{\lambda}(x)|^2 |u_{\lambda}(y)|^2 \,dx \,dy \right| \le C, \notag \\
&\left| \intR \intR \big( G(x-y)-G(x) \big) |u_{\lambda}(x)|^2 \rho(y) \,dx\,dy \right|
\le C, \notag \\
&\hbox{but} \intR G(x) |u_{\lambda}(x)|^2 \,dx \to \infty 
\quad \hbox{as} \ \lambda \to \infty, \label{eq:7.2}
\end{align}
where $C>0$ is a constant independent of $\lambda$.

When $N=1$, it follows that $G(x)= \frac{1}{2}|x|$ 
and $\{ u_{\lambda} \}$ can be constructed by considering a function
whose mass is supported near the origin and infinity.
(See \cite[Example 4.1 and Remark 4.6]{DeLeo}.)
However in the 3D case, which yields $G(x)= \frac{1}{4\pi |x|}$, 
it cannot happen that both \ef{eq:7.1} and \ef{eq:7.2} are fulfilled.
Indeed by the Hardy inequality, we have
\[
\begin{aligned}
\intR G(x) |u(x)|^2 \,dx
&= \frac{1}{4\pi} \intR \frac{ |u(x)|^2}{|x|} \,dx
\le \frac{1}{4\pi} \left( \intR \frac{|u(x)|^2}{|x|^2} \,dx \right)^{\frac{1}{2}}
\left( \intR |u(x)|^2 \,dx \right)^{\frac{1}{2}} \\
&\le C \| \nabla u \|_{L^2(\R^N)} \| u \|_{L^2(\R^N)}.
\end{aligned}
\]
Moreover as we have shown in Lemma \ref{lem:3.1}, 
the energy functional on $\R^3$ is always bounded from below,
regardless of the size of $\rho$.
Therefore the only possibility for the non-existence is that 
the strict sub-additivity does not hold when $\mu< \| \rho \|_{L^1(\R^N)}$.
Moreover as we have observed in this paper,
in the 3D problem, 
it is rather natural to measure the size of $\rho$ by $L^{\frac{6}{5}}$-norm,
which makes us to conjecture that the non-existence result may be obtained
if $\| \rho \|_{\frac{6}{5}}$ is large.

\bigskip
\noindent {\bf Acknowledgements.}
The authors are grateful to anonymous referees for 
carefully reading the manuscript and providing us valuable comments,
especially about the proof of Lemma \ref{lem:3.6}.
The second author has been supported by JSPS KAKENHI Grant Numbers 
JP21K03317, JP24K06804.
This paper was carried out while the second author was staying
at the University of Bordeaux. 
The second author is very grateful to all the staff of the University of Bordeaux 
for their kind hospitality.


\end{document}